\documentclass[onesided]{article}

\usepackage{amsmath}
\usepackage{amssymb}
\usepackage{pxfonts}
\usepackage{graphicx}
\usepackage{theorem}
\usepackage{pifont}
\usepackage{color}
\usepackage[all]{xy}
\usepackage{hyperref}
\usepackage{tikz}
\usepackage{tkz-euclide}
\usepackage{enumitem}
\usetkzobj{all}
\usepackage{bbm}
\usepackage{enumitem}

\theoremheaderfont{\fontfamily{pzc}\bfseries\large}
\newtheorem{theorem}{Theorem}[section]
\newtheorem{lemma}[theorem]{Lemma}
\newtheorem{proposition}[theorem]{Proposition}

\newtheorem{corollary}[theorem]{Corollary}
\newtheorem{definition}[theorem]{Definition}

{\theorembodyfont{\rmfamily}
\newenvironment{proof}{{\flushleft \emph{Proof}:}}{\hfill\ding{110}}
\newenvironment{comment}{{\flushleft \fontfamily{pzc}\bfseries\large Comment:}}{}
\newenvironment{comments}{{\flushleft \fontfamily{pzc}\bfseries\large Comments:}}{}
\newenvironment{remark}{{\flushleft \fontfamily{pzc}\bfseries\large Remark:}}{}
\newenvironment{example}{{\flushleft {\fontfamily{pzc}\bfseries\large Example}:}}{\hfill 
$\blacktriangle\blacktriangle\blacktriangle$}
\newenvironment{proof1}[1]{{\flushleft \emph{#1}:}}{\hfill\ding{110}}

\newcommand{\secref}[1]{Section~\ref{#1}}
\newcommand{\figref}[1]{Figure~\ref{#1}}
\newcommand{\thmref}[1]{Theorem~\ref{#1}}
\newcommand{\defref}[1]{Definition~\ref{#1}}

\newcommand{\propref}[1]{Proposition~\ref{#1}}
\newcommand{\lemref}[1]{Lemma~\ref{#1}}
\newcommand{\corrref}[1]{Corollary~\ref{#1}}

\newcommand{\brk}[1]{\left(#1\right)}          
\newcommand{\Brk}[1]{\left[#1\right]}          
\newcommand{\BRK}[1]{\left\{#1\right\}}        

\newcommand{\beq}{\begin{equation}}
\newcommand{\eeq}{\end{equation}}

\providecommand{\R}{\mathbb{R}}

\newcommand{\Textand}{\qquad\text{ and }\qquad}

\newcommand{\g}{\mathfrak{g}}

\newcommand{\euc}{\mathfrak{e}}

\newcommand{\dist}{\operatorname{dist}}
\newcommand{\len}{\operatorname{len}}
\newcommand{\supp}{\operatorname{supp}}
\newcommand{\M}{\mathcal{M}}

\newcommand{\pl}{\partial}
\newcommand{\e}{\varepsilon}
\newcommand{\dis}{\operatorname{dis}}
\newcommand{\diam}{\operatorname{diam}}
\newcommand{\Leb}{\operatorname{Leb}}

\newcommand{\ind}{{\mathbbm{1}}}

\newcommand{\boldd}{\boldsymbol{d}}
\newcommand{\calF}{{\mathcal F}}
\newcommand{\calS}{{\mathcal S}}
\newcommand{\bbZ}{{\mathbb Z}}
\newcommand{\bbN}{{\mathbb N}}
\newcommand{\calR}{{\mathcal R}}
\newcommand{\calA}{{\mathcal A}}
\newcommand{\calB}{{\mathcal B}}
\newcommand{\calC}{{\mathcal C}}

\newcommand{\homega}{\widehat{\omega}}
\newcommand{\hOmega}{\widehat{\Omega}}

\newcommand{\prob}{\mathbf{P}}
\newcommand{\Exp}{\mathbf{E}}

\newcommand{\revision}[1]{{#1}}
\newcommand{\rerevision}[1]{{#1}}

\numberwithin{equation}{section}

\begin{document}

\title{Non-metricity in the continuum limit of randomly-distributed point defects}
\author{Raz Kupferman\footnote{Institute of Mathematics, The Hebrew University.
}, Cy Maor\footnotemark[1] $\,$ and Ron Rosenthal\footnote{Department Mathematik, ETH Z\"urich.} }
\date{}
\maketitle

\begin{abstract}

We present a homogenization theorem for isotropically-distributed point defects, by considering a sequence of manifolds with increasingly dense point defects. 
The loci of the defects are chosen randomly according to a weighted Poisson point process, making it a continuous version of the first passage percolation model.  
We show that the sequence of manifolds converges to a smooth Riemannian manifold, while the Levi-Civita connections converge to a non-metric connection on the limit manifold. Thus, we obtain rigorously the emergence of a non-metricity tensor, which was postulated in the literature to represent continuous distribution of point defects.
\end{abstract}

\tableofcontents


\section{Introduction}

The study of defects in solids with imperfections is a longstanding theme in \revision{materials-science}.
One of the prototypical crystalline defects are point defects (see e.g.~\cite{Kro81, Kro90}).
In crystalline materials, point defects may be caused by vacancies, interstitials, or impurities.
In amorphous materials, point defects can be viewed as either a contraction or a dilatation of local equilibrium distances between adjacent material elements.
Assuming that a defect-free body is modeled by a smooth Euclidean manifold (see e.g.~\cite{KMS15}), a body containing isotropic (i.e.~ball-shaped) point defects is modeled by a subset $D\subset \R^d$ endowed with a Riemannian metric of the form,
\beq
\label{eq:def_point_defects}
	\g_R = \phi \cdot \euc,\qquad 
	\phi(x) =
	\begin{cases}
		\xi^2 & |x-x_i| \leq \frac{1}{R} \text{ for some } i=1,\dots,m\\
		1 & |x-x_i| > \frac{1}{R} \text{ for every } i=1,\dots,m,
	\end{cases}
\eeq
where $\euc$ is the Euclidean metric on $\R^d$, the points $x_1,\ldots,x_m$ are the centers of the defects, $|\cdot|$ is the Euclidean norm, $1/R$ is the radius of a defect, and $\xi$ is the dilatation factor. In this work we focus on defects of vacancy type, hence $\xi<1$.
\revision{$x_i$ can be thought of as the loci of "missing", or "smaller" atoms in the material; "neighboring" atoms thus occupy the vacant location and get closer.}

\revision{In mechanics and materials science, a major theme is the modeling of} materials that contain distributed defects (see e.g.~\cite{Nol67,Wan67}). In continuum models, bodies with distributed defects are modeled as smooth manifolds, in which the singularities are smoothed out (or homogenized) and their density is represented by an additional geometric field.
For example, bodies with distributed dislocation-type defects have been modeled since the 1950s as a Riemannian manifold endowed with a metrically-consistent, {\em non-symmetric} flat connection (e.g.~\cite{Nye53,BBS55,BS56}). In this model, the density of the dislocations is represented by the torsion-field of the connection.

A model for distributed point defects has been much more elusive. 
It has been suggested that bodies with distributed point defects could be modeled as Riemannian manifolds with a flat, symmetric, {\em non-metric} connection (e.g.~\cite{Kro81} p. 300--304 and \cite{MR02}). There is, however, a big difference between the continuum models of dislocations and point defects. Since the 1950s, there has been a clear rationale---even if not a rigorous derivation---relating dislocations to torsion.
We are not aware of a similar rationale relating point defects to non-metricity.
In the words of Kr\"oner in his seminal review (\cite{Kro81} p. 304):
\begin{quote}
{\em
We are, however, completely aware of the fact that these ... identifications [of point defects and non-metricity] have not the same degree of certainty as the ... identification ... of dislocations and torsion.
}
\end{quote}

In this paper we present a rigorous analysis of the homogenization of point defects. Similarly to the homogenization of edge-dislocations \cite{KM15, KM15b}, we obtain a manifold endowed with a non-metric connection as a limit of manifolds with distributed defects.
Specifically, we consider a sequence of manifolds with increasingly-dense point defects. As the density of the defects tends to infinity, the sequence of locally-smooth manifolds with singularities converges to a smooth Riemannian manifold.
The Levi-Civita connection on each manifold in the sequence is, wherever it is defined, the Levi-Civita connection of $\R^d$. When the distribution of the point defects is not uniform, the Levi-Civita connection of $\R^d$ is {\em inconsistent} with the limit metric, i.e.~has a non-zero non-metricity tensor. This is the source of non-metricity in the limit.
A surprising feature of our result is that the limit metric (and hence the non-metricity tensor) is not that expected from volume vs. length considerations, see \secref{sec:Length_volume_inconsistency}.

In this paper, we investigate isotropic distributions of isotropic point defects. That is, the distribution is locally invariant to rotations (isotropy of the distribution), and the defects are ball shaped (isotropy of the defects).
A natural way to achieve such a distribution is to randomly select the loci of the defects using a weighted Poisson point process. The precise model, which is detailed in \secref{sec:model}, turns out to be a continuous version of the first passage percolation model on the Euclidean lattice, thus making it an interesting probabilistic model on its own; see Subsection \ref{sec:model} for more details.

The structure of this paper is as follows: 
In \secref{sec:overview} we give a rather informal presentation of the main results (without getting into the probabilistic details), and discuss their geometric and \revision{materials-science/mechanics} consequences. \secref{sec:overview} is the most relevant for the geometric and material-science-oriented reader.
In \secref{sec:model} we describe the probabilistic model for the distribution of the point defects and discuss its connections to the probabilistic literature. After a list of definitions and notations in \secref{sec:notation_and_definitions}, we state the main results in \secref{sec:subsection_main_results}.
Since the proofs are rather technical, we provide in  \secref{sec:proof_sketch} a sketch of the proof of our main results. The detailed proofs are presented in  Sections \ref{sec:Uniform_dist_of_points}--\ref{sec:proof_main_thm}.

\paragraph{Acknowledgements}
We are very grateful to Marcelo Epstein for suggesting us the question of homogenization of defects, and to Pavel Giterman for fruitful discussions.
\revision{We are also grateful to the anonymous referees, who pointed out some errors, and helped us to improve the readability of the paper.}
The first author is partially supported by the Israel Science Foundation and by the Israel-US Binational Foundation. The third author is partially supported by an ETH fellowship.

\section{Setting and main results}

\subsection{Overview of the results and discussion}
\label{sec:overview}

Let $d\ge 2$ be the dimension and let $\xi\in (0,1)$ be the dilatation factor of the point defects. Ignoring momentarily the probabilistic details, our main result (\thmref{thm:Main_theorem}) is  roughly as follows:
\begin{quote}
{\em There exists a $u_*>0$ and a continuous monotonically-decreasing function $\eta: [0,u_*)\to (0,1]$ such that the following holds:
Let $D\subset \R^d$ be a compact $d$-dimensional manifold with corners. 
Let $u:D\to [0,u_*)$ be a continuous function. 
Let $(D,\g_R)_{R>0}$ be a family of manifolds containing $R^d \cdot \int_D u $ point defects of intensity $\xi$ and radius $1/R$, randomly distributed in $D$ with distribution $u$. 
Then $(D,\g_R)$ converges (in the Gromov-Hausdorff sense) as $R\to\infty$ to the Riemannian manifold $(D, (\eta\circ u) \cdot \euc)$.}
\end{quote}

This result holds also if $\xi=0$ (i.e., if point defects correspond to the removal of a ball and the identification of its boundary as a single point). The case of $\xi=0$ involves however semi-metrics, hence to simplify the presentation we will only consider in this subsection the case $\xi>0$.

The next subsections discuss geometric and material-science consequences of our main theorem.

\subsubsection{Non-metricity}

The Riemannian (Levi-Civita) connection for each of the manifolds $(D,\g_R)$ coincides with the Euclidean Levi-Civita connection $\nabla$ of $\R^d$, whenever it is defined, i.e., everywhere except for the boundaries of the defects. Thus, as $R\to\infty$, the connection converges \revision{(in $L^\infty$)} to the Euclidean connection.
If $u$ is not constant (i.e.~the point defects are distributed non-uniformly), then the limiting Riemannian metric is not Euclidean, hence the limit connection does not coincide with the Levi-Civita connection of the limit metric. In other words, parallel transport with respect to the limit connection $\nabla$ is not an isometry in the limit manifold.

In fact, one can consider the convergence of Riemannian manifolds with connections,
\[
(D,\g_R,\nabla)\to (D,(\eta\circ u) \cdot \euc,\nabla),
\]
in which case Riemannian manifolds with metric connections converge to a Riemannian manifold with a non-metric connection.

If $\eta\circ u$ is differentiable, then the non-metricity tensor $Q_{kij}$ of $\nabla$ is given in coordinates by
\beq
\label{eq:non_metricity}
Q_{kij} = \nabla_{k} \g_{ij} = \pl_k \g_{ij} - \Gamma^l_{ik} \g_{lj} - \Gamma^l_{jk} \g_{il} 
		= \pl_k (\eta\circ u)\, \delta_{ij} = \frac{\pl_k (\eta\circ u)}{\eta\circ u} \, \g_{ij}.
\eeq
where $\g_{ij}$ are the coordinate components of the metric $(\eta\circ u) \cdot \euc$, and  $\Gamma^l_{jk}$ are the Christoffel symbols of $\nabla$, which are identically zero.

In particular, the non-metricity tensor is diagonal with respect to the metric. 
This is consistent with the model presented in \cite{YG12} for bodies with distributed point defects: Riemannian manifolds with flat, symmetric, non-metric connections with a non-metricity tensor diagonal with respect to the metric (in \cite{YG12} such manifolds are called Weyl manifolds).

We believe that the fact that the off-diagonal components of the non-metricity tensor are zero is only a result of our choice of \emph{isotropic} point defects (i.e.~ball-shaped). For different choices of point defects we expect the emergence of non-diagonal non-metricity tensors;
see open questions below.

\subsubsection{Curvature}

An immediate corollary of our main result is that any subset of $\R^d$ endowed with a conformally-flat Riemannian metric can be obtained as a limit of Euclidean manifolds with point defects. 

In particular, the limit manifold can have non-zero curvature, even though the point defects do not carry any curvature charge. This is similar to the case of dislocations, where it is only the limit connection that is flat---not the limit metric (see \cite{KM15b}).
	
\subsubsection{Length-volume inconsistency}
\label{sec:Length_volume_inconsistency}

The Gromov-Hausdorff convergence of  $(D,\g_R)$ to $(D,(\eta\circ u) \cdot \euc)$ as $R\to\infty$ is a convergence of distance functions in metric spaces. Another property that converges as $R\to\infty$ is the measure $\nu_R$ induced by the Riemannian metric $\g_R$. It weakly converges to the measure $\mu_{\sigma\circ u}$ induced by the Riemannian metric $(\sigma\circ u)\cdot \euc$ on $D$, where
\[
\sigma(u) = \brk{e^{-u \kappa_d}+\xi^d(1-e^{-u\kappa_d})}^{1/d},
\]

and $\kappa_d=\frac{\pi^{d/2}}{\Gamma(d/2+1)}$ is the volume of the $d$-dimensional unit ball.
Indeed, for constant distribution $u$, the defects cover as $R\to\infty$ a fraction of $(1-e^{-u\kappa_d})$ of the manifold; this remains true locally for non-constant $u$.
In other words, the sequence of \revision{metric measure} spaces $(D,\dist_{R}^D, \mu_{R})$ converges in the measured-Gromov-Hausdorff topology (see \defref{def:mGH} below) to the \revision{metric measure} space $(D,\boldd_{\eta\circ u}^D,\mu_{\sigma\circ u})$, where $\dist_{R}^D$ and $\boldd_{\eta\circ u}^D$ are the intrinsic distance functions induced on $D$ by $\g_R$ and $(\eta\circ u) \cdot \euc$, respectively. 
See part 2 of \thmref{thm:Main_theorem} and \corrref{cor:mGH_convergence}  for details.

A na\"ive guess would be that $\eta$ equals $\sigma$\revision{, since then, both the limit distance function and the limit measure are derived from the same Riemannian metric ($\sigma$ is the $d$-th root of the volume reduction), like the distance functions and measures for every finite $R$}.
Even though our analysis does not yield an explicit formula for $\eta$, we show that, in fact, $\eta\revision{<}\sigma$ (part 3 of \thmref{thm:Main_theorem}) \revision{and in the case $\xi=0$, even $\eta\le \sigma^d$ (actually $\eta<\sigma^d$ is achievable, see a remark in Section~\ref{sec:eta_upper_bound})}. Hence, the limit metric and the limit measure are inconsistent with each other. This can be viewed as another type of non-metricity, not related to the connection, which, to our knowledge, has not been mentioned in the material-science literature so far.

\subsubsection{Open questions}

We conclude this section by raising several natural questions awaiting further analysis:

\begin{enumerate}

\item 
In the present work we assume that point defects are spherically-symmetric (``isotropic" point defects).
A natural question is what is the limit if one takes non-isotropic point defects, say ellipsoids. 
\revision{
In the non-isotropic case, our analysis predicts convergence to some limiting metric space; the latter is not 
conformally-Euclidean, unlike the isotropic case.
}
It is not clear, however, whether the  limit distance function is induced by a Riemannian metric (a plausible alternative would be a Finsler metric). 
		
If the limit distance function is induced by a Riemannian metric, then we expect the resulting non-metricity tensor not to be diagonal with respect to the metric as in \eqref{eq:non_metricity}. 
Either way, whether the limit metric is Riemannian or Finsler, this will show that Weyl manifolds (in the sense of \cite{YG12}) are not the most general model for distributed point defects, as suggested in \cite{YG12}.
		
\item 
A similar question arises if the defects are placed on a grid deterministically (or if the distribution is not isotropic). As the grid spacing tends to zero, we conjecture the appearance of a non-Riemannian limit metric, even if the entire structure is symmetric (say, cubic defects on a cubic grid).

\item 
This work focuses on the phenomenological question of {\em describing} bodies with distributed point defects. Another natural question is how do the non-metricity and the length-volume inconsistency manifest in the {\em mechanical}, or elastic properties of the body. 
That is, if each manifold $(D,\g_R)$ represents an elastic body with some elastic energy density related to its metric, what is the elastic energy functional in the $R\to\infty$ limit?
\end{enumerate}

\subsection{Probabilistic setting}
\label{sec:model}

Let $d\geq 2$. We consider the space of locally finite point measures,
\beq \label{eq:Om_defn}
	\Omega = \BRK{\omega=\sum_{i\geq 0} \delta_{x_i} ~:~ \begin{array}{l}
	x_i \in\R^d \text{ for all } i\geq 0 \text{ and } \omega(A)<\infty \\
	\text{for all compact } A\subset \R^d 
	\end{array}}
\eeq
with its natural $\sigma$-algebra $\calF$ generated by the evaluation maps $\omega\mapsto \omega(A)$, with $A$ running over all Borel-measurable sets in $\R^d$.

For $x\in\R^d$ and $r\geq 0$, let 
\[
\overline{B(x,r)}=\{y\in\R^d ~:~ |x-y|\le r\}
\] 
be the closed Euclidean ball of radius $r$ around $x$. 

Let $\xi\in[0,1)$. Given $\omega\in\Omega$ and $R>0$ we denote 
\[
\calS_R(\omega) =\bigcup_{x\in\textrm{supp}(\omega)} \overline{B(x,1/R)},
\] 
and define 
a Riemannian (semi-)metric on $\R^d$
\beq
\label{eq:def_g_R}
\g_R(x;\omega) =
	\begin{cases}
		\euc & x\notin \calS_R(\omega),\\
		\xi^2\cdot \euc & x\in \calS_R(\omega),
	\end{cases}
\eeq
where $\euc$ is the Euclidean metric on $\R^d$. We will often remove $\omega$ from the notation when no confusion occurs. For $\xi>0$, $\g_R$ is a Riemannian metric, and for $\xi=0$ it is a semi-metric. 
Let $\dist_R$ be the distance function induced by $\g_R$, that is
\beq \label{eq:defn_dist}
	\dist_R(x,y;\omega)=\inf\{\len_R(\gamma)~:~\gamma\in\Gamma(x,y)\},
\eeq
where $\Gamma(x,y)$ denotes all the paths from $x$ to $y$, and $\len_R(\gamma)$ is the length of $\gamma$ induced by $\g_R$. If $\xi=0$, $\dist_R$ is a semi-distance.

A note on nomenclature: the term ``metric" is commonly used in two different contexts---for a Riemannian metric on a smooth manifold and for a distance function in a metric space. Since the distinction between the two is at the heart of the present work, we will consistently call the first a metric and the second a distance.

Finally, denote by $\nu_R$ the measure on $\R^d$ induced by $\g_R$,
\beq
\label{eq:def_mu_R_R_d}
	\nu_R(A;\omega) = \Leb_d(A\setminus \calS_R(\omega)) + \xi^d \Leb_d(A\cap \calS_R(\omega)),
\eeq
where $\Leb_d$ is the $d$-dimensional Lebesgue measure and $A\subseteq \R^d$ is a Lebesgue measurable set.

The triple $(\R^d,\dist_R,\nu_R)$ is a metric measure space if $\xi>0$. In order to obtain a metric measure space for $\xi=0$, we define the equivalence relation,
\beq\label{eq:equivalence_relation}
	x \overset{\omega,R}{\sim} y \hspace{0.4cm}\Leftrightarrow \hspace{0.4cm}
	\text{$x=y$ or $x,y$ are in the same connected component of $\calS_R(\omega)$}.
\eeq
In other words, for every $x\in\mathrm{supp(\omega)}$, we identify all the points in $\overline{B(x,1/R)}$. 

For given $\omega\in\Omega$ and $R>0$ the equivalence relation yields a metric measure space $(\M_R,\dist_R,\nu_R)$, where 
\beq
		\M_R = \R^d/\overset{R}{\sim}.
\eeq
Denoting by $\pi_R:\R^d\to\M_R$ the equivalence class map associated with $\overset{R}{\sim}$ (and with a slight abuse of notation),
\begin{equation}\label{eq:dist_M_R}
	\dist_R(x,y) := \dist_R(\pi_R^{-1}(x),\pi_R^{-1}(y)),\quad \forall x,y\in\M_R.
\end{equation}
and
\begin{equation}\label{eq:measure_M_R}
	\nu_R(A) := \nu_R(\pi_R^{-1}(A)),\quad \forall \text{ Borel set $A\subset\M_R$}.
\end{equation}

Note that in \eqref{eq:dist_M_R} both $\pi_R^{-1}(x)$ and $\pi_R^{-1}(y)$ may
contain more than one element. However, the distance doesn't depend on the choice of the representatives by Definition \eqref{eq:defn_dist} of $\dist_R$. We denote by $\M_R'$ the set $\{x\in \M_R : |\pi_R^{-1}(x)| =1\}$,
where here and below $|\cdot|$ denotes the cardinality of a set.

\begin{comment} 
In simple words, given $\omega\in\Omega$ and $R>0$ the metric measure space $\M_R=\M_R(\omega)$ is  the metric measure space induced from $\R^d$ by 
identifying all the points in the closed balls $\overline{B\brk{x_i,1/R}}$, where $x_i$ are the points in the support of $\omega$.
\end{comment}

Similarly to the full space, for a path-connected Lebesgue-measurable subset $D\subset \R^d$,
we denote by $\dist_R^D$ the intrinsic distance/semi-distance induced by $\g_R$ on $D$, that is 
\begin{equation}\label{eq:defn_of_dist_R^D}
	\dist_R^D = \inf\{\len_R(\gamma) ~:~ \gamma\in \Gamma(x,y),~\gamma \subset D\},
\end{equation}
and denote by $\nu_R^D$ the restriction of $\nu_R$ to $D$.
For $\xi>0$, this yields a metric measure space $(D,\dist_R^D,\nu_R^D)$. For $\xi =0$,
we denote
\begin{equation}\label{eq:defn_of_D_R}
	D_R = D/\overset{R}{\sim},
\end{equation}
obtaining a metric measure space $(D_R, \dist_R^D,\nu_R^{D})$, 
where $\dist_R^D$ and $\nu_R^D$ are defined (with a slight abuse of notation) as the pullback of $\dist_R^D$ and $\nu_R^D$ by $\pi_R$, similarly to the definitions in \eqref{eq:dist_M_R} and \eqref{eq:measure_M_R}.

Note: in order to address at the same time the cases $\xi=0$ and $\xi>0$,
we will sometimes write $D_R$ and $\pi_R$ even if $\xi>0$. In this case $D_R=D$ and $\pi_R$ is the identity map (when $\xi>0$, all points in $D$ are distinct).

Given a function $u:\R^d\to [0,\infty)$ and $R>0$, we denote by $\prob_{u,R}$ the probability measure on $(\Omega,\calF)$ under which $\omega$ is a Poisson point process with intensity $R^d u(x)\cdot \Leb_d(dx)$, see \cite[Chapter 3]{Re87}  for details on Poisson point processes. We denote by $\Exp_{u,R}$ the corresponding expectation. Note that as $R\to\infty$, the number of point defects grows like $R^d$ whereas the volume in $\R^d$ of each point defect scales like $1/R^d$, 
which is why we expect the measure to converge.

We will show below (Lemma \ref{lem:Basic_properties_of_M_R}) that for $\xi=0$ and for every density function $u:\R^d\to [0,\infty)$ taking values in a compact set of an interval $[0,u_*)$, $u_*=u_*(d)>0$, the metric space $\M_R$ is $\prob_{u,R}$-a.s.~simply-connected and locally isometric to the Euclidean space, everywhere except for a nowhere dense set. This is also trivially true for the case $\xi>0$, in which $(\R^d,\g_R)$ is locally isometric to the Euclidean plane at every continuity point of $\g_R$.

\subsubsection{Discussion on the probabilistic model}

If the function $u:\R^d\to [0,\infty)$ is spatially invariant (i.e.~constant), one can view our model as a continuous version of first passage percolation on $\bbZ^d$, in which one associate a random weight with each edge independently according to some weight distribution $F$, see for example \cite{Ke84}. Our specific choice of points defects is related to the function $F(\xi)=p$ and $F(1)=1-p$ for a certain choice of $p$ that depends on $u$. Note however that one can easily generalize the model to general weights by sampling the dilation factor of each ball randomly according to some distribution. 

The discrete model of first passage percolation was introduced by Hammersley and Welsh \cite{HW65}, and was studied extensively since then; see e.g.~ \cite{Ke84,GK84,Ke93,BKS03} and the references therein. The continuous version of this model is based on its continuous counterpart for percolation (known as the Boolean model); see \cite{Ha85,MR96}. To the best of our knowledge, there is no existing work on a continuous version of first passage percolation.

For a general continuous function $u$, the discrete counterpart of the model is first passage percolation with independent but not identically distributed edge weights.

\subsection{Notation and definitions}
\label{sec:notation_and_definitions}

In this subsection we list notations and definitions that will be used throughout the paper. 

\subsubsection{Norms}

\begin{itemize}
\item \makebox[0.9cm]{$|\cdot|$\hfill} -- The inner-product (Euclidean) norm on $\R^d$.

\item \makebox[0.9cm]{$|\cdot|_\infty$\hfill} -- The supremum norm on $\R^d$. 

\item \makebox[0.9cm]{$\|\cdot\|_\infty$\hfill} -- The supremum norm on continuous functions. 

\end{itemize}

\subsubsection{Riemannian metrics}
\begin{itemize}

\item $\euc$ -- The Euclidean metric on $\R^d$.

\item $\g_R=\g_R(\cdot;\omega)$ -- The (random) metric on $\R^d$ induced by point defects according to \eqref{eq:def_g_R} (or equivalently \eqref{eq:def_point_defects}).

\end{itemize}

\subsubsection{Distance functions}

\begin{itemize}

\item Given a continuous function $\rho:\R^d\to (0,\infty)$, 
\begin{equation}\label{eq:defn_of_d_rho}
	\boldd_\rho = \begin{array}{l}\text{the distance function induced}\\
	\text{by the Riemannian metric $\rho(x) \cdot \euc$}.
	\end{array}
\end{equation}

\item For a path-connected compact subset $D\subset\R^d$,  
\begin{equation}\label{eq:defn_of_d_rho^D}
	\boldd_\rho^D = \begin{array}{l}
	\text{the intrinsic distance function induced} \\
	\text{by the Riemannian metric $\rho(x)\cdot \euc \text{ on } D$.}
	\end{array}
\end{equation}

\item $\dist_R=\dist_R(\cdot;\omega)$ -- The distance function induced by $\g_R$ on $\R^d$ or $\M_R$ (see \eqref{eq:defn_dist} and \eqref{eq:dist_M_R}).

\item $\dist$ -- a shortened notation for $\dist_1$ used in Sections \ref{sec:Uniform_dist_of_points}--\ref{sec:concentration_results}.

\item $\dist_R^D=\dist_R^D(\cdot;\omega)$ -- The intrinsic distance function induced by $\g_R$ on a path-connected Lebesgue-mesurable set $D\subset \R^d$, or on $D_R$ (see \eqref{eq:defn_of_dist_R^D} and \eqref{eq:defn_of_D_R}).

\item For two compact sets $A,B\subset \R^d$, we denote by $d_H(A,B)$ their Hausdorff distance with respect to the Euclidean metric on $\R^d$, 
\[
d_H (A,B) = \inf \BRK{ \e>0 ~:~ 
\begin{array}{l}
	\forall x\in A\, \exists y\in B, ~\text{s.t. } |x-y|<\e ~\text{and}\\
	\forall x\in B\, \exists y\in A, ~\text{s.t. } |x-y|<\e
\end{array}	
}.
\]

\item For compact metric spaces $X$ and $Y$, we denote by $d_{GH}(X,Y)$ their Gromov-Hausdorff distance (see \defref{def:mGH} below).

\end{itemize}

\subsubsection{Measures}

\begin{itemize}

\item $\Leb_d$ -- The Lebesgue measure on $R^d$.

\item $\mu_\rho$ -- The measure induced on $\R^d$ by the volume form of the Riemannian metric $\rho\cdot \euc$.

\item $\mu_\rho^D$ -- The restriction of $\mu_\rho$ to $D$.

\item $\nu_R=\nu_R(\cdot;\omega)$ -- The measure induced by $\g_R$ on $\R^d$ or $\M_R$ (see \eqref{eq:def_mu_R_R_d} and\eqref{eq:measure_M_R}).

\item $\nu_R^D=\nu_R^D(\cdot;\omega)$ -- The restriction of $\nu_R$ to a subset $D\subset \R^d$ (or to $D_R$).

\end{itemize}

\subsubsection{Other notations}

\begin{itemize}

\item $B(x,r)$ -- The (open) Euclidean ball of radius $r>0$ around $x\in \R^d$.

\item $\overline{A}$ -- The closure of a set $A\subset \R^d$.

\item $\diam_\euc(A) = \sup\{|x-y| ~:~ x,y\in A\}$ -- The Euclidean diameter of a set $A\subset \R^d$.

\item $S^{d-1}$ -- The Euclidean sphere $\{x\in\R^d ~:~ |x|=1\}$.

\item $SO(d)$ -- The special orthogonal group in $\R^d$ with respect to the Euclidean inner-product.

\item $\kappa_d=\frac{\pi^{d/2}}{\Gamma(d/2+1)}$ -- The volume of the $d$-dimensional unit ball.

\item By a path in $\R^d$ we will always mean a continuous parametrized path $[0,1]\to\R^d$. 

\item $\len_\euc(\gamma)$ -- The Euclidean length of a path $\gamma$.

\item $\len_R(\gamma)=\len_R(\gamma;\omega)$ -- The length of a path $\gamma$ induced by $\g_R$.

\item $\len(\gamma)$ -- a shortened notation for $\len_1(\gamma)$ used in Sections \ref{sec:Uniform_dist_of_points}--\ref{sec:concentration_results}.
.

\item $[x,y]$ -- The linear segment connecting $x$ and $y$ in $\R^d$.

\item $\text{Lip}(f)$ -- The Lipschitz constant of a continuous function $f:D\to \R$ for some $D\subset \R^d$ \revision{(with respect to the Euclidean metric)}.

\item $C(D)$ - The space of continuous functions on $D$.
\end{itemize}

\subsubsection{Measured Gromov-Hausdorff convergence}
\label{subsec:2.3.6}

\begin{definition}
\label{def:mGH}
In the following $(Z,d_Z)$ is a compact metric space and $(Z,d_Z,\mu_Z)$ is a compact \revision{metric measure} space.
\begin{enumerate}
\item	
For a function $f:(X,d_X)\to (Y,d_Y)$ between two metric spaces, we define the \textbf{distortion} of $f$ by
		\beq
		\label{eq:def_distortion}
			\dis f = \sup_{x,y\in X} |d_X(x,y) - d_Y(f(x),f(y))|. 
		\eeq
A function $f:D(f)\subset X \to Y$ is called an \textbf{$\e$-approximation} if $\dis f <\e$, $B_X(D(f),\e) = X$ and $B_Y(\text{Im}(f),\e) = Y$, where $B_X(A,\e)$ is the $\e$-neiborhood in $X$ around a subset $A$, with a similar definition for neighborhoods of sets in $Y$.

\item 
The \textbf{Gromov-Hausdorff distance} is a distance function between isometry classes of compact metric spaces. 
For the purpose of this paper, it is enough to state that $d_{GH}((X,d_X),(Y,d_Y)) < 4\e$ if there exists an $\e$-approximation $f:D(f)\subset X \to Y$. For further details see e.g.~\cite{Pet06}, Chapter 10.
		
\item 
A sequence $(X_n,d_n,\mu_n)$ of compact \revision{metric measure} spaces converges to a compact \revision{metric measure} space $(X,d,\mu)$ in the \textbf{measured Gromov-Hausdorff topology} if there exists a sequence $f_n: D(f_n)\subset X_n \to X$ of $\e_n$-approximations, with $\e_n\to0$,  such that the pushforward measures $(f_n)_\#\mu_n$ weakly converge to $\mu$. That is, for every continuous function $\Phi:X\to\R$,
\[
\lim_{n\to\infty} \int_{D(f_n)} \Phi\circ f_n \, d\mu_n = \int_X \Phi \, d\mu.
\]
\end{enumerate}
\end{definition}

\subsubsection{Remark about constants}

Throughout the paper, constants are denoted by $C$ and $c$. The dependence of constants on parameters will be denoted by brackets. For example, $C=C(d)$ implies that $C$ only depends on the dimension $d$. Note that the value of such constants may change from one line to the next. 
Numbered constants $C_1,C_2,...$ have a fixed value which is determined in their first appearance.


\subsection{Main results}
\label{sec:subsection_main_results}

Our main result is the following:

\begin{theorem}\label{thm:Main_theorem}
Let $d\geq 2$. There exists a real number $u_*>0$ depending only on $d$, and there exists a continuous, non-constant, monotonically-decreasing function $\eta:[0,u_*)\to (0,1]$ with $\eta(0)=1$ that depends only on $d$ and $\xi\in[0,1)$, such that for every compact $d$-dimensional submanifold $D\subset\R^d$ with corners and every continuous function $u:D\to [0,u_*)$:
\begin{enumerate}
\item Metric convergence:
For every $\e>0$,
\begin{equation}
\label{eq:Main_distortion}
	\lim_{R\to\infty} \prob_{u,R} \brk{ \sup_{x,y\in D} \left| \boldd_{\eta\circ u}^D(x,y) - \dist_R^D(\pi_R(x),\pi_R(y)) \right| < \varepsilon }=1.
\end{equation}
		
\item Measure convergence:
For every $\e>0$,
\begin{equation}
\label{eq:Main_measure}
\lim_{R\to\infty} \prob_{u,R} \brk{ \forall f\in W(D),\,\,  \left| \int_D f \, d\mu_{\sigma\circ u}^{D}  - \int_{D_R} f\circ\pi_R^{-1} \, d\nu_R^{D} \right| <\varepsilon }=1,
\end{equation}
where 
\[
\revision{
W(D) = \BRK{ f\in C(D) : \|f\|_\infty \le 1 \,\,\text{and}\,\, \text{Lip}(f) \le 1 },}
\]
and
\[
\sigma(u) = \brk{e^{-u \kappa_d}+\xi^d(1-e^{-u\kappa_d})}^{1/d}.
\] 
	
\item Length-volume incompatibility:
\[
\eta(u) \le e^{-u\kappa_d}+ \xi (1-e^{-u\kappa_d}), 
\]
and in particular \revision{$\eta(u)<\sigma(u)$; if $\xi=0$, then $\eta(u)\le \sigma^d(u)< \sigma$.} 
	
\item 
The following is of interest when $\xi=0$:
for every $\e>0$,
\begin{equation}
\label{eq:Main_surjective}
	\lim_{R\to\infty} \prob_{u,R} \brk{ d_H(D,\pi^{-1}_R(D'_R)) <\varepsilon }=1,
\end{equation}
where $D'_R = \{x\in D_R ~:~ |\pi_R^{-1}(x)|=1\}$.
In other words, $\pi_R^{-1}:D'_R\to D$ is asymptotically surjective.

\end{enumerate}
\end{theorem}

\thmref{thm:Main_theorem} implies the following corollary:

\begin{corollary}$~$
\label{cor:mGH_convergence}
\begin{enumerate}

\item Part 1 of \thmref{thm:Main_theorem} implies that for every $\e>0$,
\begin{equation}
\label{eq:GHconvergence}
	\lim_{R\to\infty} \prob_{u,R}\brk{d_{GH}\brk{(D_R,\dist_R^D),(D,\boldd_{\eta\circ u}^D)}<\varepsilon}=1.
\end{equation}

Moreover, for every sequence $(\omega_n,R_n,\e_n)$, where $\omega_n\in \Omega$, $R_n>0$ and $\e_n\to 0$ such that the events in \eqref{eq:Main_distortion}--\eqref{eq:Main_surjective} hold with $\e_n$ and $R_n$, the sequence of \revision{metric measure} spaces
$(D_{R_n},\dist_{R_n}^D, \nu_{R_n}^{D})(\omega_n)$ converges in the measured-Gromov-Hausdorff topology to the \revision{metric measure} space $(D,\boldd_{\eta\circ u}^D,\mu_{\sigma\circ u}^{D})$.

\item Part 3 in \thmref{thm:Main_theorem} implies that while the measures $\mu_{R_n}^{D}$ and the distance functions $\dist_{R_n}^D$ on $D_{R_n}$ are consistent, the limit measure $\mu_{\sigma\circ u}$ is strictly larger than the measure $\mu_{\eta\circ u}$ induced by the metric $(\eta\circ u)\cdot \euc$ associated with the limit distance $\boldd_{\eta\circ u}^D$.

\end{enumerate}
\end{corollary}

\revision{
\begin{comments}
\begin{enumerate}
\item The constant $u_*$ is the percolation threshold; for $u>u_*$, $\calS_1$ contains almost surely  a unique infinite connected component (see also a remark at the end of Section~\ref{sec:4.2}). 
\item The assumption of $u$ being continuous can be relaxed; for example, the theorem holds if the $d$-dimensional Hausdorff measure of the set of discontinuous points of $u$ is zero.
\item The bound $\eta\le \sigma^d$ when $\xi=0$ is not tight. With a more complicated argument it is possible to show $\eta<\sigma^d$. See a remark at the end of Section~\ref{sec:eta_upper_bound}.
\item Throughout this paper, we assume that "manifolds with corners" do not include cusps. This assumption can be presumably relaxed.
\end{enumerate}
\end{comments}
}


\section{Sketch of the proof}
\label{sec:proof_sketch}

In the remaining sections of this paper, we prove \thmref{thm:Main_theorem} and \corrref{cor:mGH_convergence}. 
For \revision{the sake of brevity}, we assume that $\xi=0$, which is the most difficult case; \revision{the proofs for $\xi>0$ are similar, and in certain parts, much simpler}.
Since the proof is long and technical, this section describe its main stages.

In Sections \ref{sec:Uniform_dist_of_points}--\ref{sec:convergene_uniform_distribution}, we consider a uniform distribution of point defects, that is, the case where $u$ is a constant function. In  \secref{sec:proof_main_thm}, we generalize the results to arbitrary continuous functions $u:D\to [0,u_*)$.

\paragraph{\secref{sec:Uniform_dist_of_points}:}
For a constant $u$, a uniform rescaling of $\R^d$ enables the introduction of a natural coupling between the probability measures $(\mathbf{P}_{u,R})_{R>0}$ using a single measure $\mathbf{P}_u$ that is used in the rest of the proof for the uniform case.

We show the existence of a sub-critical regime, i.e., a constant $u_*=u_*(d)>0$, such that for $u\in[0,u_*)$, the process $\mathbf{P}_u$-a.s.~does not percolate. In this regime, $\M_R$ is ``nice''; it is $\mathbf{P}_u$-a.s.~a simply-connected metric space, which is locally isometric to the Euclidean space but for a nowhere-dense set.
From this point onward we work only in the subcritical regime $u<u_*$.

We then use the subadditive ergodic theorem to obtain the existence of a limit distance function in $\M_R$. We show that for every $u\in[0,u_*)$ there exists an $\eta(u)\in [0,1)$ such that 
\[
\prob_u\brk{\lim_{R\to\infty} \dist_R(\pi_R(x),\pi_R(y)) = \eta(u)\cdot |x - y|,\quad\forall x,y\in\R^d}=1.
\]
Moreover,  this limit is uniform in every compact $K\subset \R^d$.
This establishes the $\mathbf{P}_u$-a.s.~Gromov-Hausdorff convergence of $(K_R, \dist_R)$ to $(K,\mathbf{d}_{\eta(u)})$.

Note, however, that this is not quite the metric convergence we want in \eqref{eq:Main_distortion} for a uniform $u$, since we want to prove that $(K_R, \dist^D_R)$ converges to $(K,\mathbf{d}^D_{\eta(u)})$ (i.e.~the intrinsic distances converge and not only the induced distances). 
Moreover, we need to prove that $\eta(u)>0$ for $u\in (0,u_*)$. This is done in the next sections.

\paragraph{\secref{sec:large_deviations}:} 
In this section we prove large deviation results for the distance function in $\M_R$. An immediate corollary is that $\eta(u)>0$ for $u\in (0,u_*)$.
The key idea here is to exploit the independence structure of the Poisson point process, manifested in the BK inequality, in order to show that  distances do not deviate significantly from their expected value.

\paragraph{\secref{sec:concentration_results}:}
In this section we prove several results regarding the geometry of geodesics in $(\M_R,\dist_R)$, and properties of the function $\eta$ that controls the limit metric.

By using the concentration results of Section \ref{sec:large_deviations}, we prove that geodesics in $\M_R$ are, with high probability, very close to straight lines between their endpoints. 

Using a coupling between the probability measures $(\mathbf{P}_u)_{u\geq 0}$, we prove that $\eta(u)$ is a continuous, monotonically-nondecreasing function,  and give an upper bound on the value of $\eta(u)$ that implies that $\eta<\sigma$ (which proves part 3 of \thmref{thm:Main_theorem}). 

Finally, we use the ergodicity of the model to prove that $\nu_R([0,1]^d)$ converges $\mathbf{P}_u$-a.s.~to $\mu_{\sigma(u)}([0,1]^d)$ (which is a step towards proving the measure convergence in part 2 of \thmref{thm:Main_theorem}).

\paragraph{\secref{sec:convergene_uniform_distribution}:}
In this section we prove Parts 1 and 4 of \thmref{thm:Main_theorem} for uniform distributions on a convex, compact $d$-dimensional manifold with corners $D\subset \R^d$.
The idea behind the proof of Part 1 is as follows: from \secref{sec:Uniform_dist_of_points} we know that $\mathbf{P}_u$-a.s.~$(D_R, \dist_R)$ Gromov-Hausdorff converges to $(D,\mathbf{d}_{\eta(u)})$. Since $D$ is convex and $u$ is constant, $\mathbf{d}_{\eta(u)}= \mathbf{d}^D_{\eta(u)}$ on $D$, so we only need to replace $(D_R, \dist_R)$ with $(D_R, \dist^D_R)$.
We do so by showing that the identity mapping between $(D_R, \dist_R)$ and $(D_R, \dist^D_R)$ has $\mathbf{P}_u$-a.s.~vanishing distortion. Here we use the result from \secref{sec:concentration_results} that geodesics in $\M_R$ are very close to straight lines, which implies that $\dist_R$-geodesics between points tend to remain within $D$, i.e.~they are also $\dist^D_R$-geodesics with high probability.

\paragraph{\secref{sec:proof_main_thm}:} 
In this section, we conclude the proof of \thmref{thm:Main_theorem} and \corrref{cor:mGH_convergence} for a general continuous distribution $u$ of defects over a $d$-dimensional manifold with corners $D\subset \R^d$. 

The idea is the following: we partition $D$ into small cubes, such that $u$ is approximately constant in each cube. 
We show that the results of Sections \ref{sec:concentration_results} and \ref{sec:convergene_uniform_distribution} apply approximately to each of the cubes, in the sense that for every cube $\square$, when $R$ is large, the distance between the \revision{metric measure} spaces $(\square_R, \dist^\square_R,\nu^\square_R)$ and $(\square, \mathbf{d}^\square_{\eta\circ u}, \mu^\square_{\sigma\circ u})$ is bounded with high probability by the variation of $u$ in $\square$.

Finally, we glue the cubes together and obtain results for the whole manifold $D$.
For measure convergence, the gluing is straightforward.
For metric convergence, we use the control on the distortion in each cube and a bound on the number of cubes each geodesic crosses to control the Gromov-Hausdorff distance between $(D_R, \dist^D_R)$ and $(D,\mathbf{d}^D_{\eta(u)})$ (in a similar way as in \cite{KM15,KM15b}).

\section{Uniform distribution of point defects}\label{sec:Uniform_dist_of_points}

In this section, as well as in the three to follow, we study the simplest version of the model: we assume that the function $u:\R^d \to (0,\infty)$ is constant. In this section, we prove the existence of a subcritical regime (Lemma \ref{lem:Basic_properties_of_M_R}), and obtain our first main result regarding distances in the manifolds $\M_R$ (Theorem \ref{thm:distances_in_M_R}). Many of the results obtained in these sections are adaptations to the continuous setting of the results obtained in   \cite{HW65,Ke84} for the discrete case.

\subsection{Rescaling for uniform distributions}

The parameter $R>0$ is a scaling factor that affects both the density of the point defects and their magnitude, or volume. In particular, to every $R$ corresponds a different probability measure $\prob_{u,R}$.   
For spatially invariant $u$, the intensity measures are translationally-invariant. By a uniform rescaling of space, we may obtain a probabilistic model that \revision{does not} depend on $R$, and in particular allows us to construct a natural coupling of the measures $\mathbf{P}_{u,R}$. 

This is done as follows: construct a Poisson point process on $\R^d$ with intensity $u\cdot \Leb_d(dx)$, i.e., set $R=1$. We denote the probability measure by $\prob_u = \prob_{u,1}$, and the corresponding expectation by $\Exp_u = \Exp_{u,1}$. Two points in $\M_1$ are identified, $x \overset{\omega,1}{\sim} y$, if they are in the same ball of  radius $R=1$ centered at a point in the support of $\omega$. We denote by $\dist = \dist_1$ the corresponding distance function on $\R^d$, which can also be written as
\begin{equation} \label{eq:defn_dist2}
	\dist(x,y;\omega)=\inf\BRK{\sum_{i=1}^{N} |x_i-y_{i}| ~:~
	\begin{array}{l}
	 N\in\bbN,\,\,\, x_1\revision{\overset{1}{\sim}} x,\,\,\, x_{i+1}\revision{\overset{1}{\sim}} y_i, \,\,\, y_N\revision{\overset{1}{\sim}} y
	 \end{array}
	 }.
\end{equation}

Given $R>0$, let $T_R:\R^d\to\R^d$ be defined by $T_R(x)=Rx$.  Then, \revision{$(T_{1/R})_\#\omega$} is a Poisson point process with intensity $R^d u\cdot \Leb_d(dx)$, namely \revision{$\prob_{u,R} = (T_{1/R})_\# \prob_u$}. 
\revision{Here, $(T_{1/R})_\#$ is the push forward of measures by $T_{1/R}$, where the push forward of $\prob_u$ is the one induced by the push forward of $\omega$. 
That is
\[
(T_{1/R})_\#\omega = \sum_i \delta_{T_{1/R}x_i}
\qquad
\text{for}\qquad
\omega = \sum_i \delta_{x_i},
\]
and,
\[
(T_{1/R})_\# \prob_u(A) = \prob_u(\{\omega : \rerevision{(T_{1/R})_\#}\omega \in A\})
\]
 }
Similarly, we identify $\M_R$ with $\M_1$ via the scaling $T_{1/R}$. The distance function $\dist_R$ on $\R^d$ is derived from the distance function $\dist = \dist_1$ by the following relation,
\beq 
\label{eq:dist_R_coupled}
	\dist_R(x,y;\omega)=\frac{1}{R}\dist\brk{\revision{\pi_1(R \,\pi_R^{-1}(x)),\pi_1(R \,\pi_R^{-1}(y))};\omega},\qquad\forall x,y\in\M_R.
\eeq
\rerevision{
Note that there is a slight abuse of notation here: strictly speaking, $\dist_R(x,y;\omega)$ as defined in \eqref{eq:dist_R_coupled} coincides with $\dist_R(x,y;(T_{1/R})_\# \omega)$, as defined in \eqref{eq:defn_dist}.
However, the distribution of \eqref{eq:dist_R_coupled} with respect to $\prob_u$ is the same as the distribution of \eqref{eq:defn_dist} with respect to $\prob_{u,R}$.
In Sections \ref{sec:Uniform_dist_of_points}--\ref{sec:convergene_uniform_distribution}, where the above coupling is used, we use $\dist_R(x,y;\omega)$ in the sense of \eqref{eq:dist_R_coupled}, so that $\prob_u$ can be used for all values of $R$.
}

\subsection{The sub-critical regime}
\label{sec:4.2}

We start our analysis by proving the existence of a subcritical regime: 

\begin{lemma}
\label{lem:Basic_properties_of_M_R}
For $\omega\in\Omega$ let $\calS(\omega) = \calS_1(\omega)=\bigcup_{x\in\supp(\omega)} \overline{B(x,1)}$. Then
\begin{enumerate}[label={(\arabic*)}]
\item 
Each of the connected components of $\calS(\omega)$ is $\prob_u$-almost surely closed. 

\item 
There exists a constant $u_*>0$, depending only on $d$, such that for all $u\in(0,u_*)$ 
\begin{equation}\label{eq:no_infinite_cluster}
	\prob_u\brk{\begin{array}{l}
	\text{The set }\calS~\text{doesn't contain an}\\
	\text{infinite connected component}
	\end{array}
	}=1
\end{equation}
and for every $u>u_*$
\begin{equation}\label{eq:unique_infinite_cluster}
	\prob_u\brk{\begin{array}{l}
		\text{The set }\calS~\text{contains a unique}\\
		\text{infinite connected component}
		\end{array}
		}=1.
\end{equation}
In particular, \revision{recalling that $\xi=0$,} whenever $u\in(0,u_*)$ the simply connected metric space $\M_R$ is $\prob_u$-a.s.~locally isometric to the Euclidean space, up to a nowhere dense set.
\end{enumerate}
\end{lemma}	

\revision{
\begin{comment}
If $\xi>0$, $\M_R$ is almost everywhere locally isometric to a Euclidean space, with a scaling constant depending on whether the point is the interior of $\calS$ or in the complement of $\calS$. 
\end{comment}
}

\begin{comment}
The subset $\M_R'$ of $\M_R$ that is locally isometric to the Euclidean space can be identified with $\R^d\setminus\calS(\omega)$, \revision{not only as sets, but also as Riemannian manifolds}. However, they are not globally isometric. Note that $\M_R'$ is $\prob_u$-a.s.~not connected.
\end{comment}

\begin{proof}
The fact that each connected component of $\calS(\omega)$ is $\prob_u$-almost surely closed follows from the fact that with $\prob_u$-probability one, every bounded set contains only finitely many points of $\supp(\omega)$. This implies that \revision{$\prob_u$-almost surely, the point process does not have accumulation points, hence} the complement of $\calS(\omega)$ is $\prob_u$-almost surely open. 

The existence of  $u_*>0$ such that $\calS(\omega)$ doesn't contain $\prob_u$-a.s.~infinite clusters for every $u\in (0,u_*)$ and contains $\prob_u$-a.s.~an infinite cluster  for $u>u_*$ is the content of \cite{Ha85}; see also Theorem~3.3 in \cite{MR96}. 
The uniqueness of the infinite cluster for $u>u_*$ can be found in \cite{MR96} Theorem~3.6.

Since $\M_R$ is obtained from $\R^d$ by a similarity transformation and an identification of points in simply-connected subsets, $\M_R$ is simply connected. As proved above, $\R^d\setminus \calS(\omega)$, which is identical to $\M_R'$ up to a similarity transformation, is open. It follows that $\M_R'$ is locally isometric to Euclidean space. 

Finally, we need to show that for $u\in (0,u_*)$ the set $\M_R\backslash \M_R'$ is nowhere dense. This follows from the fact that any compact subset of $\R^d$ contains $\prob_u$-a.s.~only finitely many points in $\textrm{supp}(\omega)$, hence every compact subset of $\M_R$ contains only finitely many points in $\M_R\backslash \M_R'$.
\end{proof}

\revision{
\begin{comment}
It can be shown that for $u>u_*$ and $d=2$, in every box of sufficiently large radius $R$, the distance of every point from the unique infinite component is at most $O(\log R)$. As a result, the limiting distancs $\dist_R$ between every pair of points is zero.
\end{comment}
}

\subsection{Distances in $\M_R$ in the uniform case}

The main result of this section proves the existence of a limit distance function in $\M_R$. The precise statement is as follows:

\begin{theorem}\label{thm:distances_in_M_R} 
For every $u\in[0,u_*)$ there exists an $\eta(u)\in[0,1]$ such that 
\beq
	\prob_u\brk{\lim_{R\to\infty} \dist_R(\pi_R(x),\pi_R(y)) = \eta(u)\cdot |x - y|,\quad\forall x,y\in\R^d}=1.
	\label{eq:4.5}
\eeq
The limits also exist in the $L^1(\Omega,\prob_u)$ sense: for every $x,y\in\R^d$,
\begin{equation}
\lim_{R\to\infty} \Exp_u\Brk{\left|\dist_R(\pi_R(x),\pi_R(y)) - \eta(u) \cdot |x - y|\right|} = 0.
\end{equation}
Furthermore, the convergence of the distance function is uniform in  every compact $K\subset\R^d$ (and in particular in $S^{d-1}$; \revision{this particular case is used below}),
\begin{equation}
\prob_u\brk{\begin{array}{c}
\forall \epsilon>0\,\, \exists R_0\,\, \text{such that}\,\, \forall R>R_0,\,\,\\
\sup_{x,y\in K} \left|\dist_R(\pi_R(x),\pi_R(y)) - \eta(u)\cdot |x - y|\right|<\epsilon 
\end{array}} = 1
\end{equation}
and 
\begin{equation}
	\lim_{R\to\infty} \sup_{x,y\in K}\Exp_u\Brk{\left|\dist_R(\pi_R(x),\pi_R(y)) - \eta(u)\cdot |x - y|\right|} = 0.
\end{equation}
\end{theorem}

In Sections \ref{sec:large_deviations}--\ref{sec:concentration_results} below we prove certain properties of $\eta(u)$, and in particular that $\eta(u)>0$ for every $u\in[0,u_*)$.

\revision{
Equation~\eqref{eq:4.5} and similar equations hereafter should be interpreted as follows: 
\[
\prob_u\brk{\BRK{\omega\in\Omega ~:~ 
\lim_{R\to\infty} \dist_R(\pi_R(x),\pi_R(y);\omega) = \eta(u)\cdot |x - y|,\quad\forall x,y\in\R^d}}=1,
\]
where $\dist_R$ is understood as in \eqref{eq:dist_R_coupled}.
}

An immediate corollary of \thmref{thm:distances_in_M_R} is:

\begin{corollary}
For every $u\in[0,u_*)$ and compact $K\subset\R^d$, the sequence of metric spaces $(K_R,\dist_R)$ defined by \eqref{eq:defn_of_D_R} $\prob_u$-a.s.~Gromov-Hausdorff converges to $(K,\boldd_{\eta(u)})$, 
where $\eta(u)$ should be considered as a constant function on $\R^d$.
\end{corollary}

Note that this convergence is with respect to the induced distances $\dist_R$ on $K_R$ and not with respect to the intrinsic metric $\dist_R^K$ on $K_R$, defined in \eqref{eq:defn_of_dist_R^D}. Proving the convergence of the intrinsic metric is more involved and requires more assumptions on $K$. This is done in \secref{sec:convergene_uniform_distribution}.


\subsection{Proof of Theorem \ref{thm:distances_in_M_R}}

We start by reformulating \thmref{thm:distances_in_M_R} using the relation
between $\dist$ and $\dist_R$ and the definition of $\pi_R$:

\begin{theorem}[Rephrasing of Theorem \ref{thm:distances_in_M_R}]\label{thm:distance_in_the_plane}
For every $u\in[0,u_*)$ there exists $\eta(u)\in[0,1]$ such that 
\beq \label{eq:distance_in_the_plane}
	\prob_u\brk{\lim_{R\to\infty} \frac{\dist(Rx,Ry)}{R} = \eta(u)\cdot |x-y|,\quad\forall x,y\in\R^d}=1. 
\eeq
The limits also exist in $L^1(\Omega,\prob_u)$: for every $x,y\in\R^d$,
\[
\lim_{R\to\infty} \Exp_u\Brk{\left|\frac{\dist(Rx,Ry)}{R} - \eta(u)\cdot |x - y|\right|} = 0.
\] 
Furthermore, the convergence of the $\dist(Rx,Ry)/R$ is uniform \revision{over $x,y$} in every compact $K\subset\R^d$ (and in particular in $S^{d-1}$;  \revision{this particular case is used below}).
\end{theorem}	

The proof of Theorem \ref{thm:distance_in_the_plane} is separated into several parts and starts with the observation that the system $(\Omega,\calF,\prob_u)$ is ergodic with respect to translations. 

\begin{lemma}\label{lem:ergodicity}
	For $x\in\R^d$ define $\tau_x : \Omega \to \Omega$ by $\tau_x\brk{\sum_{i\geq 0} \delta_{x_i}} = \sum_{i\geq 0 } \delta_{x_i-x}$. Then, for every $x\in\R^d\setminus \{0\}$ the quartet $(\Omega,\calF,\prob_u,\tau_x)$, defines a translation invariant ergodic system. 
\end{lemma}

\begin{proof}
	See for example \cite{MR96} Proposition 2.6. 
\end{proof}

Next, we prove the existence of the limit in \eqref{eq:distance_in_the_plane} for $x=0$.

\begin{lemma} \label{lem:distance_in_the_plane_lemma_1}
	Let $u\in[0,u_*)$. For every $y\in \R^d$ the limit
	\beq
		\rho_u(y) = \lim_{R\to\infty}\frac{\dist(0,Ry)}{R}
	\eeq
	exists $\prob_u$-a.s.~and in $L^1(\Omega,\prob_u)$.
\end{lemma}

\begin{proof}
If $u=0$, then $\omega(\R^d)=0$ $\prob_u$-a.s., i.e., $\dist(x,y)=|x-y|$ with $\prob_u$-probability one, which implies that
\[
\lim_{R\to\infty} \frac{\dist(0,Ry)}{R} = |y|,
\qquad \text{$\prob_u$-a.s.}
\]

Thus the result holds with $\rho_0(y)=|y|$. 

We turn to the case $u>0$. For $y=0$ the statement is trivial, so fix $y\in\R^d\backslash \{0\}$ and define for $0\leq m< n$
\beq
	Y_{m,n} = \dist(my,ny).
\eeq
The triangle inequality for $\dist$ implies that 
\beq\label{eq:subadditivity_of_Y_m,n}
	Y_{0,n} \leq Y_{0,m} + Y_{m,n},\quad \forall 0\leq m< n.
\eeq
Thus we are in a good position to use Kingman's subadditive ergodic theorem \cite{Ki73}. More specifically we will exploit Liggett's version \cite{Li85}, which states that if $(X_{m,n})_{0\leq m<n}$ are nonnegative random variables such that 
\begin{enumerate}[label={(\arabic*)}]
  \setlength\itemsep{0em}
	\item $X_{0,n} \leq X_{0,m}+X_{m,n}$ for all $0<m<n$,
	\item $\{X_{nk,(n+1)k} ~:~n\geq 1\}$ is stationary and ergodic for each $k\geq 1$,
	\item the law of $\{X_{m,m+k}~:~k\geq 1\}$ is independent of $m\geq 1$, and
	\item \revision{$X_{0,1}$ has finite expectation,} $E[X_{0,1}]<\infty$, 
\end{enumerate} 
then the limit $\lim_{n\to\infty}X_{0,n}/n$ exists almost surely and in $L^1$ and \revision{almost surely} equals 
\[
	\inf_{n>0}\frac{E[X_{0,n}]}{n}=\lim_{n\to\infty}\frac{E[X_{0,n}]}{n}<\infty.
\]

Taking $X_{m,n}=Y_{m,n}$,  $(1)$ is given by \eqref{eq:subadditivity_of_Y_m,n}. For $(2)$ and $(3)$ note that 
\beq
	\dist(u+z,v+z;\omega) = \dist(u,v;\tau_z\omega),\quad \forall u,v,z\in \R^d,
\eeq
and therefore $Y_{m,n} = Y_{0,n-m}\circ \tau_y^m$. Condition $(2)$ and $(3)$ then hold by the translation invariance and the ergodicity of the law $\prob_u$ under the shift $\tau_y$, see Lemma \ref{lem:ergodicity}. As for $(4)$, it follows immediately from the fact that $Y_{0,1}=\dist(0,y)\leq |y|$.

Thus, the family $\{Y_{m,n}\}_{0\leq m<n}$ satisfies conditions (1)--(4) and
\[
\lim_{n\to\infty} n^{-1}\dist(0,ny),
\] 
which we denote by $\rho_u(y)$, \revision{converges} $\prob_u$-a.s.~and in $L^1$. We then prove that 
\beq 
	\rho_u(y)=\lim_{R\to\infty}\frac{\dist(0,Ry)}{R},\quad \text{$\prob_u$-a.s}.
\eeq 
Indeed, $\rho_u(y)=\lim_{R\to\infty} \frac{\dist(0,\lfloor R\rfloor y)}{\lfloor R\rfloor} $  and since $\dist(\lfloor R\rfloor y, Ry) \leq |\lfloor R\rfloor y - Ry| \leq |y|$, it follows that 
\[
	\left|\frac{\dist(0,Ry)}{R} - \frac{\dist(0,\lfloor R\rfloor y)}{\lfloor R\rfloor}\right| \leq \frac{2|y|}{R}.
\]
This completes the proof.
\end{proof}

\begin{lemma}\label{lem:distance_in_the_plane_lemma_2}
	The function $\rho_u$ in Lemma \ref{lem:distance_in_the_plane_lemma_1} satisfies
	\beq
		\rho_u(\alpha y) = \alpha \rho_u (y), \quad \forall \alpha \in (0,\infty), y\in \R^d,
	\eeq
	and
	\beq
		\rho_u(y+z) \leq \rho_u(y)+\rho_u(z),\quad \forall y,z\in\R^d.
	\eeq
\end{lemma}

\begin{proof}
The positive homogeneity follows from the existence of the limit since 
\[
	\rho_u(\alpha y) = \lim_{R\to\infty}\frac{\dist(0,R\alpha y)}{R} = 	\lim_{R\to\infty}\alpha \frac{\dist(0,R\alpha y)}{R\alpha} = \alpha \rho_u(y). 
\]
For the triangle inequality note that by the translation invariance of $\prob_u$
\beq
\begin{aligned}
		\Exp_u[\dist(0,R(y+z))] &\leq \Exp_u[\dist(0,Ry)]+\Exp_u[\dist(Ry,R(y+z))]\\
		& =  \Exp_u[\dist(0,Ry)] + \Exp_u[\dist(0,Rz)].
\end{aligned}
\eeq
Diving both sides by $R$, taking the limit $R\to\infty$ and using the $L^1$ convergence of $\dist(0,Ry)/R$ gives the required inequality. 
\end{proof}

\begin{lemma}\label{lem:distance_in_the_plane_lemma_3}
	The function $\rho_u$ in Lemma \ref{lem:distance_in_the_plane_lemma_1} is invariant under the action of $SO(d)$.  
\end{lemma}

\begin{proof}
Since $\prob_u$ is invariant under the action of $SO(d)$ it follows that for every two points $y_1,y_2\in \R^d$ with the property that there exists $\calR\in SO(d)$ such that $y_2= \calR y_1$, 
\[
\dist(0,y_2) \overset{d}{=} \dist(0,y_1).
\]	
Thus, $\rho_u$ inherits the symmetries of $SO(d)$ as the $L^1(\Omega,\prob_u)$ limit Lemma \ref{lem:distance_in_the_plane_lemma_1}.
\end{proof}

\begin{corollary}
For $u\in [0,u_*)$,
the function $\rho_u$ in Lemma \ref{lem:distance_in_the_plane_lemma_1} is of the form
\beq \label{eq:defn_of_eta_u}
	\rho_u(y) = \eta(u)\cdot |y|
\eeq
for some $\eta(u) \in [0,1]$.
\end{corollary}

\begin{proof}
This follows from the positive homogeneity,  the sub-additivity and the isotropy of $\rho_u$ proved in 
Lemmas  \ref{lem:distance_in_the_plane_lemma_2} and \ref{lem:distance_in_the_plane_lemma_3}.
\end{proof}

Next, we consider general $x,y\in\R^d$:
 
\begin{lemma}\label{lem:dist_result_general_x_y}
Fix $u\in[0,u_*)$. Then, for every $x,y\in\R^d$ 
\[
\lim_{R\to\infty}\frac{\dist(Rx,Ry)}{R}=\eta(u)\cdot |x-y|
\]
\revision{converges} $\prob_u$-a.s.~and in $L^1$. 
\end{lemma}

\begin{proof}
For every $x,y\in\R^d$ and $\omega\in\Omega$ the relation $\dist(x,y;\omega) = \dist(0,y-x,\tau_x\omega)$ holds and therefore 
\[
\frac{\dist(Rx,Ry;\omega)}{R} = \frac{\dist(0,R(y-x);\tau_x\omega)}{R},
\]
which implies that the limit $R\to\infty$ exists $\prob_u$-a.s., is in $L^1(\Omega,\prob_u)$ and equals $\eta(u)\cdot |x-y|$. 
\end{proof}

\begin{proof1}{Proof of Theorem \ref{thm:distance_in_the_plane}}
By Lemmas \ref{lem:distance_in_the_plane_lemma_1}--\ref{lem:dist_result_general_x_y},
\revision{$\dist(Rx,Ry)/R$ converges as
$R\to\infty$} both $\prob_u$-a.s.~and in $L^1$ for every fixed pair of points $x,y\in\R^d$. It remains to verify that the limit exists $\prob_u$-a.s.~simultaneously for all pairs of points $x,y\in\R^d$. 

To this end, let $\varepsilon \in (0,2\pi)$ and let $(v_i)_{i=1}^N$ with $N=\lceil  c(d)/\varepsilon^d \rceil$ be a set of points on the unit sphere $S^{d-1}$ that form an $\varepsilon/2$-net for $S^{d-1}$ ($c(d)$ is a constant that depends only on $d$). 
Since the set $(v_i)_{i=1}^N$ is finite, it follows from \lemref{lem:dist_result_general_x_y} that
\begin{equation}\label{eq:proof_of_dist_result_1}
\prob_u\brk{\lim_{R\to\infty}\frac{\dist(Rv_i,Rv_j)}{R}=\eta(u)\cdot |v_i-v_j|,\quad \forall 1\leq i,j\leq N}=1.
\end{equation}
Given $x,y\in\R^d\backslash\{0\}$, there exist $1\leq i,j\leq N$ such that 
\begin{equation}
	\left|\hat{x}-v_i\right|\leq \varepsilon\text{ and } 	\left|\hat{y}-v_j\right|\leq \varepsilon,
\end{equation}
where $\hat{x}=x/|x|$ and $\hat{y}=y/|y|$. By the triangle inequality, 
\begin{equation}
\begin{aligned}
	 \left|\frac{\dist(Rx,Ry)}{R} - \frac{\dist(R|x| v_i,R|y| v_j)}{R}\right| 
	\leq &\frac{\dist(Rx,R|x|v_i)}{R} + 	\frac{\dist(Ry,R|y|v_i)}{R}\\
	\leq & \frac{|Rx-R|x|v_i|}{R} + 	\frac{|Ry-R|y|v_i|}{R} \\ 
	\leq & \varepsilon(|x|+|y|),
\end{aligned}
\end{equation}
and  
\begin{equation}
	\left|\eta(u)\cdot |x-y| - \eta(u)\cdot ||x|v_i-|y|v_j|\right| \leq \varepsilon(|x|+|y|).
\end{equation}
Combining both estimates, it follows from \eqref{eq:proof_of_dist_result_1}  that
\[
	\prob_u\brk{\limsup_{R\to\infty}\left|\frac{\dist(Rx,Ry)}{R}-\eta(u)\cdot |x-y|\right|\leq 2\varepsilon(|x|+|y|), \quad \forall x,y\in\R^d\setminus\{0\}}=1.
\]
The cases $x=0$, $y=0$ can be included by a similar argument, except that no approximation for $0$ is needed. 
Since this holds for every $\varepsilon>0$, the limit as $R\to\infty$ exists $\prob_u$-a.s. 

To justify the uniformity over compact sets, note that there are only finitely many $v_i$'s for a fixed $\varepsilon>0$ and that for compact sets we have a uniform bound on the Euclidean norm of both $x$ and $y$. The same applies for the $L^1(\Omega,\prob_u)$ convergence. 
\end{proof1}

\section{Large deviation results}\label{sec:large_deviations}

In this section we prove large deviation results for the distance in $\M_R$. As an immediate corollary we obtain that $\eta(u) >0$ for every $u\in [0,u_*)$. 

The main result of this section is the following.

\begin{theorem}\label{thm:Chemical_distance}
For every $u\in [0,u_*)$ and every $\varepsilon>0$ there exists a positive constant $c_1$, depending only on $d$, $u$ and $\varepsilon$, such that for large enough $R$,
\begin{equation}\label{eq:LDP}
	\prob_u\brk{\exists x\in\R^d~:~|x|=R,~|\dist(0,x)-\eta(u)R|> \varepsilon R}< e^{-c_1 R}.
\end{equation}
\end{theorem}
The proof of Theorem \ref{thm:Chemical_distance} follows the ideas developed for the discrete case by Kesten \cite{Ke84}. 

In addition, we will need a large deviation result for the existence of very long geodesics  in the Euclidean sense. In order to state it we need another definition: Geodesics in $\M_R$ can be identified with geodesics in $\R^d$ with respect to the semi-distance function $\dist$. Such geodesics are highly degenerate, as there is nothing that limits their behavior in $\calS(\omega)$. For $x,y\in\R^d$ we denote by $\Gamma_0(x,y)$ the set of geodesics between $x$ and $y$ with respect to $\dist$  that minimize the Euclidean distance inside $\calS(\omega)$. We will call such paths \emph{true geodesics} (see \figref{fig:truegeo}). 

\begin{figure}
\begin{center}
\includegraphics[height=2in]{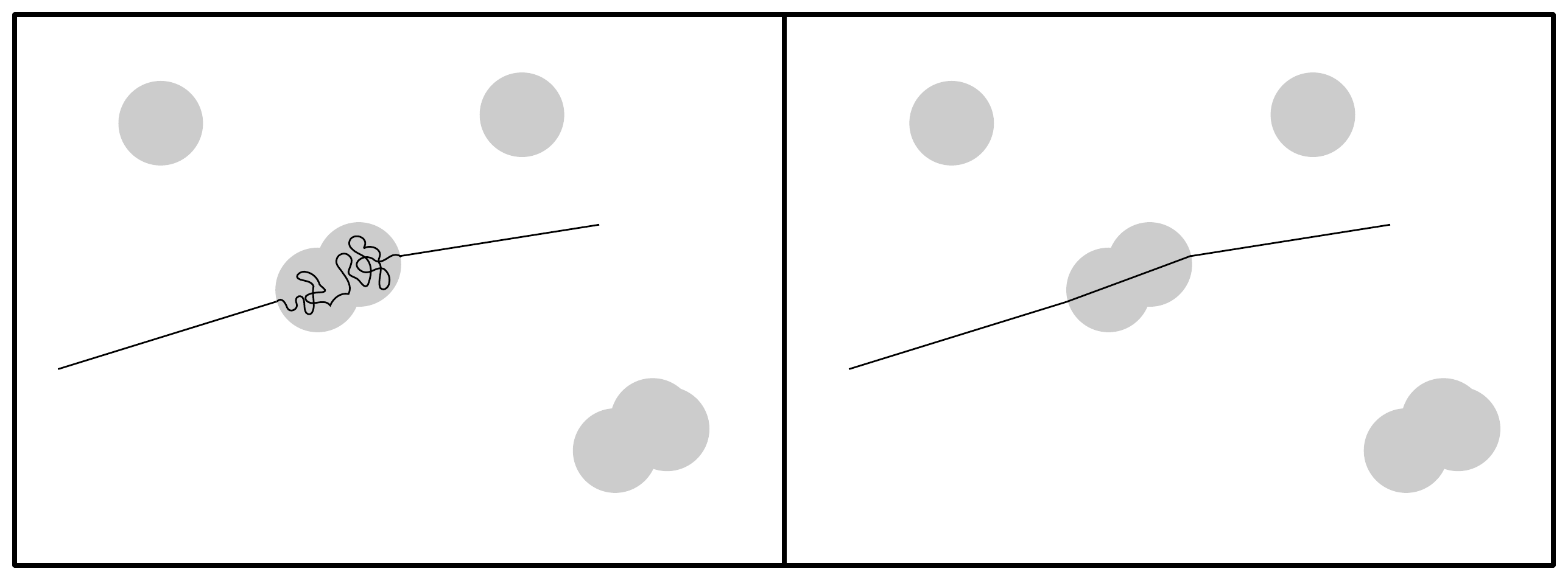}
\end{center}
\caption{The path on the left is a geodesic in $\M_R$, where the gray regions correspond to $\calS_R(\omega)$. Note that the part in $\calS_R$ has zero length. The path on the right is the corresponding true geodesic.}
\label{fig:truegeo}
\end{figure}

\begin{theorem}
\label{thm:too_long_Euclidean_distance_is_very_unlikely}

For every $u\in[0,u_*)$ there exist positive constants $C_2,c_3$ \revision{and $\alpha\gg1$}, depending only on $d$ and $u$, such that 
\begin{enumerate}[label={(\arabic*)}]
\item For every $R>0$,
\begin{equation}\label{eq:towards_eta_u_positive}
\prob_u\brk{\begin{array}{c} \exists \text{ a true geodesic path starting at } 0\\ \text{ such that } \len_\euc(\gamma)\geq R \text{ and } \len(\gamma)<\frac{1}{\alpha} \len_\euc(\gamma) \end{array}}\leq C_2e^{-c_3R}.
\end{equation}

\item 
For every $x,y\in \R^d$ and for every $R>\alpha|x-y|$
\begin{equation}
\prob_u\brk{\exists \gamma \in \Gamma_0(x,y)~:~\len_\euc(\gamma)> R }\leq C_2e^{-c_3R}.
\end{equation}
\end{enumerate}
\end{theorem}

As a corollary we obtain:
\begin{corollary}\label{cor:eta_u_is_positive}
	$\eta(u)$ is strictly positive for every $u\in[0,u_*)$. 
\end{corollary}

\begin{proof}
Let $A_R$ be the event in \eqref{eq:towards_eta_u_positive}. By Part (1) of Theorem \ref{thm:too_long_Euclidean_distance_is_very_unlikely},  
\[
	\prob_u(A_R) \le C_2 e^{-c_3 R},\quad \forall R>0.
\]
Since the Euclidean length of a true geodesic connecting $0$ and $Re_1$ is at least $R$, for $\omega\in A_R^c$,
\[
	\dist(0,Re_1) \ge \frac{1}{\alpha} R,
\]

Thus, for large enough $R$
\[
	\Exp_u[\dist(0,Re_1)] \geq 	\Exp_u[\dist(0,Re_1)\cdot \ind_{A^c_R}] \geq \frac{1}{\alpha} R\cdot \prob_u(A_R^c)\geq \frac{1}{2\alpha}R,
\]
hence
\[
	\eta(u) = \lim_{R\to\infty}\frac{\Exp_u[\dist(0,Re_1)]}{R} \geq \frac{1}{2\alpha}>0.
\]
\end{proof}

\revision{
\begin{comment}
A conjecture in percolation theory is that
\[
\lim_{u\nearrow u_*} \eta(u) = 0.
\]
For $d=2$, an adaptation of \cite[Lemma 11.12]{Gri99} to our setting will show that this conjecture holds. 
In higher dimension this a famous open problem.
\end{comment}
}

The proofs of Theorem \ref{thm:Chemical_distance} and Theorem \ref{thm:too_long_Euclidean_distance_is_very_unlikely} are quite technical. The main idea behind the proof is to exploit the independence structure of the Poisson point process, manifested in the BK inequality, in order to show that geodesics whose lengths deviate from the expected distance in the sense of \eqref{eq:LDP}, contain sufficiently many disjoint sub-paths (which are roughly independent) whose total length deviates significantly from its expected value. Such an event is highly unlikely due to large deviations results for independent random variables.

Since the proofs in this section are technical and since one can use the theorems as ``black boxes'' in the rest of the paper, the reader might wish to skip the rest of this section in a first reading.

\subsection{The BK inequality}

In this subsection we state the continuous version of the well-known BK inequality for product measures. For this, we need some additional definitions.

There is a natural partial ordering on $\Omega$, which we denote by $\preceq$, under which $\omega \preceq \omega'$ if and only if $\supp(\omega)\subseteq \supp(\omega')$. Using it one can define increasing and decreasing events in $\calF$. An event $A\in\calF$ is said to be increasing (respectively decreasing)  if for every $\omega \preceq \omega'$, $\omega\in A$ implies $\omega'\in A$ (i.e., $A$ is closed under increasing support). 

For any bounded Borel set $Y\subset\R^d$, define the set
\[
	\omega_Y = \supp(\omega)\cap Y,
\]
and for $\omega\in\Omega$ and $Y$ as above let
\[
	L(\omega,Y)=\left\{\omega'\in\Omega ~:~ \supp(\omega)\cap Y \subseteq \supp(\omega')\cap Y
	\right\}.
\]

In words, the event $L(\omega,Y)$ contains all configurations that inside $Y$ are larger than $\omega$. We say that an  event $A$ is an increasing event on $Y$ if $\omega\in A$ 
implies that $L(\omega,Y)\subseteq A$.

\begin{definition}
Let $A$ and $B$ be two increasing events on a bounded Borel set $Y$. Then
\[
A \circ B = \left\{\omega\in \Omega ~:~ \begin{array}{l}
	\text{ there are disjoint sets }V,W\subset \R^d \text{ such that } \\ 
	V \text{ and } W \text{ are finite unions of rational}\\ 
	\text{cubes and } L(\omega,V)\subset A, L(\omega,W)\subset B
	\end{array}
\right\},
\]
where by a rational cube we mean an open $d$-dimensional cube with rational coordinates. When $A\circ B$ occurs, we say that $A$ and $B$ occur disjointly. 
\end{definition}

\begin{example}
To \revision{illustrate} this definition, let $d=2$, let $Y=[0,1]^2$, and let $A$ be the event that there exists a path in $Y$ connecting the left boundary of $Y$ to its right boundary, whose length is less than $x$. Clearly, $A$ is an increasing event, as increasing the support of $\omega$ can only shorten paths.  For $\omega\in A$ and $V\subset Y$, it is generally not true that $L(\omega,V)\subset A$; $L(\omega,V)\subset A$ only if there exists a path in $V$, \revision{such that its length is not more than $x-\alpha$, where $\alpha$ is the sum of the Euclidean distances of the path's end points from the left and right boundaries of $Y$.} The event $A\circ A$ occurs if there exist two disjoint paths connecting the left boundary of $Y$ to its right boundary, whose length is at most $x$.
\end{example}

\begin{theorem}[BK inequality] \label{thm:BK_inequality}
	Suppose $Y$ is a bounded Borel set in $\R^d$ and $A,B$ are two increasing events on $Y$. Then for every $u>0$
\[
	\prob_u(A\circ B) \leq \prob_u(A)\prob_u(B).
\]
\end{theorem}

A proof of this inequality in a more general setting can be found in \cite{MR96}, Theorem 2.3.

\subsection{Key proposition}

The proofs of Theorem \ref{thm:Chemical_distance} and Theorem \ref{thm:too_long_Euclidean_distance_is_very_unlikely} exploit the invariance of the law of $\prob_u$ under translations and rotations. This implies that we only need to take care of large deviation results for the distance between the origin and points of the form $Re_1$, where $e_1=(1,0,\ldots,0)$.  We therefore restrict ourselves to the above case and define the following:

For $r\in\R$ let
 \[
H_r=\{x\in\R^d~:~\left<x,e_1\right>=r\}. 
\]
Given $r<s$ and a path $\gamma$ we write 
\[
	H_r < \gamma <H_s
\]
if all points of $\gamma$, except possibly its endpoints, lie strictly between the hyper-planes $H_r$ and $H_s$.

The following variants on the distance will stand in the core of the proofs. For $N,M>0$ define the random variables,
\begin{equation}\label{eq:LD_defn_of_s_M_N}
	s_{M,N} = \inf\left\{\len(\gamma) ~:~ 
	\begin{array}{l}
		\gamma \text{ is a path from } \{0\}\times [0,N]^{d-1}  \\ 
		\text{to } H_M \text{ such that } H_0 < \gamma < H_M
	\end{array}
	\right\}
\end{equation}
and
\begin{equation}\label{eq:LD_defn_of_hat_s_M_N}
	\widehat{s}_{M,N} = \inf\left\{\len(\gamma) ~:~ 
	\begin{array}{l}
		\gamma \text{ is a path from } \{0\}\times [0,N]^{d-1}\text{ to } H_M \text{ such}\\
		\text{that with the exception of its endpoints,} \\
		\gamma \subset (0,M)\times [-4M,4M]^{d-1}
	\end{array}
	\right\}.
\end{equation}
It follows from the definitions of $s_{M,N}$ and $\widehat{s}_{M,N}$ that
\begin{equation}\label{eq:LD_comparing_s_and_hat_s}
	s_{M,N} \leq \widehat{s}_{M,N}.
\end{equation}

The main estimate used in the proofs of both theorems is stated in the next proposition:

\begin{proposition}\label{prop:LD_key_prop}
Let $\brk{X_q(M,N)}_{q\geq 0}$ and $\brk{\widehat{X}_q(M,N)}_{q\geq 0}$ be sequences of independent random variables having the same distribution as $s_{M,N}$ and $\widehat{s}_{M,N}$. For every two integers $2\leq N\leq M\leq R/2$ and every real number $x\geq 0$,
\begin{equation}
\begin{aligned}
	\prob_u\brk{\dist(0,Re_1) < x}  &\leq 
	\sum_{Q \geq \frac{R}{M+N}-1} \brk{2d\brk{16\frac{M}{N}}^d}^Q \prob_u\brk{\sum_{q=0}^{Q-1} \widehat{X}_q(M,N) <x}	 \\ 
	& \leq 	\sum_{Q \geq \frac{R}{M+N}-1} \brk{2d\brk{16\frac{M}{N}}^d}^Q \prob_u\brk{\sum_{q=0}^{Q-1} X_q(M,N) <x}.	
\end{aligned}
\end{equation}
\end{proposition}

\begin{proof}
	Since $s_{M,N}\leq \widehat{s}_{M,N}$, the second inequality is immediate. Assume $\gamma:[0,1]\to\R^d$ is a simple path from the origin to $H_R$. We choose a sequence of points $(x_0,x_1,\dots,x_Q)$ along $\gamma$ as follows: Define $x_0=\gamma(0)=0$. Assume that $(x_0,x_1\ldots,x_q)$ have already been chosen such that $x_i=\gamma(t_i)$ with $0=t_0<t_1<\ldots<t_q$. Then we define 
\begin{equation}\label{eq:main_prop_defn_of_t_q}
	t_{q+1} = \min\{t\in (t_q,1] ~:~ |\gamma(t)-x_q|_\infty = M+N\}	,\quad x_{q+1}=\gamma(t_{q+1}),
\end{equation}
provided such time $t$ exists. If no such $t$ exists, i.e., $|\gamma(t)-x_q|_\infty < M+N$ for all $t\in (t_q,1]$, then we set $Q=q$ and stop the process.

It follows from the definition of the points $x_q$ that 
\begin{equation}\label{eq:bound_on_Q_main_prop}
	Q \geq \frac{R}{M+N}-1.
\end{equation}
Indeed, the distance between $\langle x_q,e_1 \rangle$ and $\langle x_{q+1},e_1 \rangle$ is at most $M+N$ and $\langle x_Q,e_1\rangle \geq R-(M+N)$. 
%

Next, we analyze the path within the time interval $[t_q,t_{q+1}]$ for $0\leq q\leq Q-1$. Since, by definition, $|x_q-x_{q+1}|_\infty = M+N$, it follows that there exist $j=j(q)\in\{1,2,\ldots,d\}$ and $\sigma =\sigma(q)\in \{-1,1\}$ such that 
\[
	\langle x_{q+1} - x_q , e_{j(q)} \rangle = \sigma(q) (M+N).
\]

We introduce the hyperplanes,
\[
H^j_r = \{x\in\R^d ~|~ \langle x,e_j \rangle = r\},
\qquad j=1,\dots,d, \qquad r\in\R,
\]
so that, in particular, $H^1_r = H_r$.
The path $\gamma|_{[t_q,t_{q+1}]}$ is strictly restricted between the hyperplanes
\[
H(q) = H^{j(q)}_{\langle x_q,e_{j(q)}\rangle}
\Textand
H'(q) = H^{j(q)}_{\langle x_q,e_{j(q)}\rangle+\sigma(q)(M+N)}\,\,.
\]
Since the distance between these hyperplanes is $M+N$, we can find two hyperplanes $H''(q)$ and $H'''(q)$ at a distance $M$ one of the other, that are strictly between $H(q)$ and $H'(q)$ (see \figref{fig:main_prop}). Further, we denote by $[t''(q),t'''(q)]$ the sub-interval  of $[t_q,t_{q+1}]$ in which the path $\gamma$ contains a unique crossing between the two hyperplanes $H''(q)$ and $H'''(q)$.

\begin{figure}
\begin{center}
\includegraphics[height=5in]{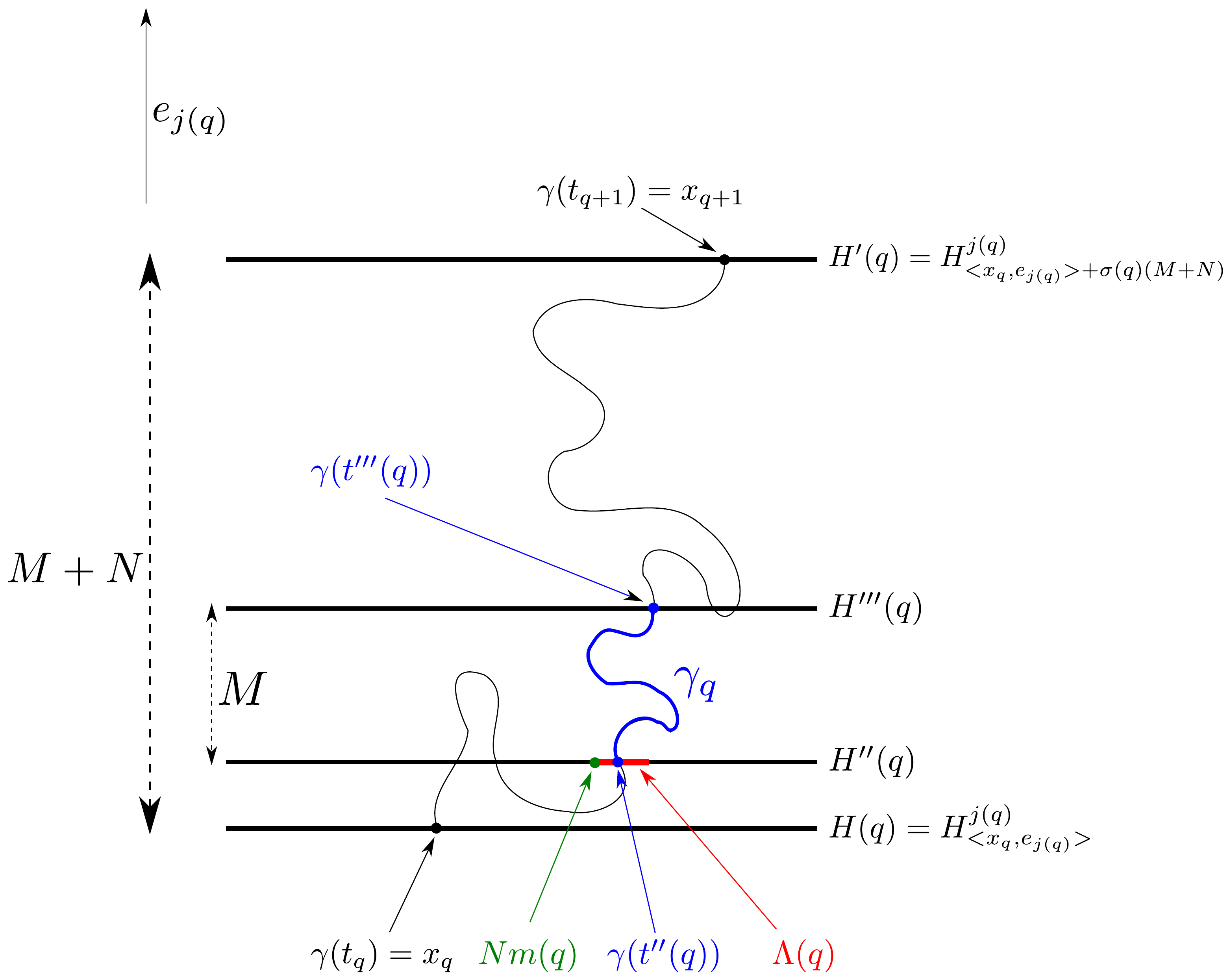}
\end{center}
\caption{Illustration of the times $t_q,t_{q+1},t''(q),t'''(q)$, the points $x_q,x_{q+1},\gamma(t''(q)),\gamma(t'''(q))$, the hyperplanes $H(q),H'(q),H''(q),H'''(q)$, the box $\Lambda(q)$ with its corner $Nm(q)$ and the path $\gamma_q$ crossing from $\Lambda(q)$ to the $H'''(q)$. }
\label{fig:main_prop}
\end{figure}

Specifically, let $m_{j(q)}(q)$ be the integer $m$ such that $H^{j(q)}_{mN}$ is the hyperplane bounded between $H(q)$ and $H'(q)$ that is the closest to $H(q)$. \revision{Explicitly,} 
\[
m_{j(q)}(q) = \left\lfloor \frac{\langle x_q,e_{j(q)}\rangle}{N}\right\rfloor + \frac{1}{2}(\sigma(q)+1)	.
\]
We then define
\[
H''(q) = H^{j(q)}_{mN}
\Textand
H'''(q) = H^{j(q)}_{mN+\sigma(q)M},
\]
and let $[t''(q),t'''(q)]\subset [t_q,t_{q+1}]$ be the subinterval defined by 
\[ \begin{array}{ll}
	t''(q) & = \max\left\{t\in [t_q,t_{q+1}) ~:~ \gamma(t) \in H''(q)\right\},\\
	\\
	t'''(q) &=\min\left\{t\in (t''(q),t_{q+1}) ~:~\gamma(t) \in H'''(q)\right\}.
	\end{array}
\]
Finally, denote by $\gamma_q$ the path $\gamma$ restricted to the time interval $[t''(q),t'''(q)]$. 

We are interested in a bound on paths starting within a $(d-1)$-dimensional box of side length $N$ within the hyperplane $H_0$, see \eqref{eq:LD_defn_of_hat_s_M_N}. We denote by $m(q)=(m_1(q),m_2(q),\ldots,m_d(q))$ the unique point in $\bbZ^d$ satisfying 

\[
\gamma(t''(q)) \in \Lambda(q) = H''(q) \cap \brk{m(q)N + [0,N)^d}.
\]
where $\Lambda(q)$ is a $(d-1)$-dimensional box of side length $N$ within the hyperplane $H''(q)$.

Exploiting all the  above definitions, the path segments $\gamma_q$ satisfy the following properties:
\begin{itemize}
	\item \revision{The images of} $\gamma_q$ \revision{in $\R^d$} are pairwise disjoint.
	\item $\gamma_q$ connects $\Lambda(q)\subset H''(q)$ to $H'''(q)$.
	\item The path $\gamma_q$ lies strictly between $H''(q)$ and $H'''(q)$,  except for its endpoints. 
	\item The path $\gamma_q$ is contained in the box 
	\[
	B(q) =  N m (q) + [-4M,4M]^{j(q)-1} \times [0,\revision{\sigma(q)}M] \times [-4M,4M]^{d-j(q)},
	\]
	where we used the fact that $N+M\le 2M$.
\item The total length of the paths $\gamma_q$ satisfies 
\[
	\sum_{q=0}^{Q-1} \len(\gamma_q) \leq \len(\gamma).
\]
\end{itemize}

Fix $x>0$, fix $Q\ge  \frac{R}{M+N}-1$ and fix $(j(q),\sigma(q),m(q))$, $q=0,\dots,Q-1$. 
We denote by $A(j,\sigma,m,x)$ the event that there exists a piecewise-linear simple path $\gamma$ containing disjoint segments $\gamma_q\subset B(q)$ crossing the box $B(q)$ in the $j(q)$ direction from $\Lambda(q)\subset H''(q)$ to $H'''(q)$, with $\sum_{q=0}^{Q-1} \len(\gamma_q) \leq x$.

Since every true geodesic is piecewise-linear, we conclude that the event
\[
\{\dist(0,Re_1) <x \}
\]
occurs only if there exists a $Q\geq \frac{R}{M+N}-1$ and  there exist $(j(q),\sigma(q),m(q))$, $q=0,\dots,Q-1$ such that $A(\revision{j,\sigma,m},x)$ occurs, \revision{therefore},
\beq \label{eq:LD_key_prop_eq_1}
\prob_u(\dist(0,Re_1) <x) \le \sum_{Q\ge  \frac{R}{M+N}-1} \sum_{(j(q),\sigma(q),m(q))}
\prob_u(A(j,\sigma,m,x)).
\eeq

The event $A(j,\sigma,m,x)$ is characterized by the existence of a path whose properties are specified over disjoint segments. Denote by
$A_q(j,\sigma,m,x)$
the event that there exists a piecewise-linear simple path $\gamma$ containing a segment $
\gamma_q\subset B(q)$ crossing the box $B(q)$ in the $j(q)$ direction from 
$\Lambda(q)\subset H''(q)$ to  $H'''(q)$ and satisfies $\len(\gamma_q)<x$.

Then,
\begin{equation}
\begin{split}
& \prob_u(A(j,\sigma,m,x)) \\
=&  \prob_u \Bigg(\bigcup_{\tiny{\begin{array}{c}
r_{0},r_{1},\ldots,r_{Q-1}\in\mathbb{Q}_{+}\\
\sum_{q=0}^{Q-1}r_{q}<x
\end{array}}} \hspace{-5mm}A_0(j,\sigma,m,r_0) \circ A_1(j,\sigma,m,r_1) \circ \dots 
 \circ A_{Q-1}(j,\sigma,m,r_{Q-1})\Bigg).
\end{split}
\end{equation}

Since the paths $\gamma_q$ are disjoint and piecewise linear, the conditions of the BK inequality are satisfied and we conclude that 
\[
\prob_u(A(j,\sigma,m,x)) \leq \sum_{\tiny{\begin{array}{c}
r_{0},r_{1},\ldots,r_{Q-1}\in\mathbb{Q}_{+}\\
\sum_{q=0}^{Q-1}r_{q}<x
\end{array}}} \prod_{q=0}^{Q-1}\prob_u(A_q(j,\sigma,m,r_q)).
\]

Noting that for a given configuration $(j,\sigma,m)$ the minimal length of a path connecting $\Lambda(q)$ to $H''(q)$ inside $B(q)$ has the same distribution as $\widehat{s}_{M,N}$, we conclude that
\[
\begin{split}
\prob_u(A(j,\sigma,m,x)) &\leq \hspace{-0.4cm} \sum_{\tiny{\begin{array}{c}
r_{0},r_{1},\ldots,r_{Q-1}\in\mathbb{Q}_{+}\\
\sum_{q=0}^{Q-1}r_{q}<x
\end{array}}} \prod_{q=0}^{Q-1}\prob_u(\widehat{X}_q(M,N)\leq r_q) 
 = \prob_u\brk{\sum_{q=0}^{Q-1}\widehat{X}_q(M,N)< x},
\end{split}
\]
which combined with \eqref{eq:LD_key_prop_eq_1}, yields
\begin{equation} \label{eq:LD_key_prop_eq_2}
	\prob_u(\dist(0,Re_1)<x) \leq \sum_{Q\geq\frac{R}{M+N}-1} \brk{\begin{array}{c}\text{number of choices}\\
	\text{for }(j,\sigma,m)\end{array}}\cdot \prob_u\brk{\sum_{q=0}^{Q-1}\widehat{X}_q(M,N)< x}.
\end{equation}

To complete the proof, we need to show that the number of ways to choose the triplets $(j,\sigma,m)$ is bounded by $(2d(16M/N)^d)^Q$.
To this end, assume that $(j,\sigma,m)$ has already been chosen. By the definition of the times $t''(q)$, $t''(q+1)$, $t_q$ and $t_{q+1}$ and the points $m(q)$ and $m(q+1)$,
\[
\begin{aligned}
N |m(q+1)- m(q)|_\infty & \leq 
|N m(q+1)-\gamma(t''(q+1))|_\infty + 
|\gamma(t''(q+1)) - \gamma(t_{q+1})|_\infty \\
&  + |\gamma(t_{q+1})-\gamma(t_q)|_\infty + 
|\gamma(t_q)-\gamma(t''(q))|_\infty \\
&+ |N m(q)-\gamma(t''(q))|_\infty \\
& \leq  3M+5N.
\end{aligned}
\]

Since $m(q+1)\in \bbZ^d$, it follows that there are at most $\brk{6\frac{M}{N} + 10}^d$ ways to choose $m(q+1)$ given $m(q)$. Moreover, there are at most $d$ choices for $j(q+1)$ and $2$ choice for $\sigma(q+1)$, hence    given $(j(q),\sigma(q),m(q))$ there are at most 
\[
2d\brk{6\frac{M}{N} + 10}^d\leq 2d\brk{16\frac{M}{N}}^d
\]
choices for $(j(q+1),\sigma(q+1),m(q+1))$ and in total at most 
\[
\brk{2d\brk{16\frac{M}{N}}^d}^Q
\]
choices for the whole sequence. 
\end{proof}

\subsection{Proof of Theorem \ref{thm:Chemical_distance}}

We separate the proof of Theorem \ref{thm:Chemical_distance} into two parts: a lower bound estimation and an upper bound estimation. We start with the first. 

Formally, the claim for the lower bound is that for every $u\in[0,u_*)$ and every $\e>0$ there exists a constant $c$, depending only on $d$, $u$ and $\varepsilon$, such that for $R$ large enough
\begin{equation}
	\prob_u\brk{\exists x\in\R^d~:~|x|=R,~\dist(0,x)< (\eta(u)-\varepsilon) R}< e^{-c R}.
\end{equation}

We start by showing that $\eta(u)$ can also be obtained as the limiting distance between a point and a hyperplane:

\begin{lemma}\label{lem:point_to_hyperplane_lemma_1}
For every $u\in[0,u_*)$, 
\begin{equation}
	\lim_{R\to\infty} \frac{\dist(0,H_R)}{R} = \eta(u), \quad\text{$\prob_u$-a.s.}
\end{equation}
\end{lemma}

\begin{proof}
From the definition of $\dist(0,H_R)$, 
\[
	\dist(0,H_R) \leq \dist(0,Re_1),
\]
and therefore 
\[
\limsup_{R\to\infty} \frac{\dist(0,H_R)}{R} \leq \lim_{R\to\infty} \frac{\dist(0,Re_1)}{R} = \eta(u),
\quad \text{$\prob_u$-a.s}.
\]
If $\eta(u)=0$, then there is nothing left to prove since $\dist(0,H_R)\geq 0$, hence assume that $\eta(u)>0$. Let $\omega\in\Omega$ be a realization such that 
\begin{equation}\label{eq:lemma_on_point_hyperplane_distance}
	\liminf_{R\to\infty}\frac{\dist(0,H_R)}{R} = \eta(u)-2\delta,\quad \text{ for some } \delta>0.
\end{equation}
Then, there exists an increasing sequence $R_k$ that tends to infinity, a sequence of points $z_k\in H_{R_k}$, and a sequence of paths $\gamma_k$ from $0$ to $z_k$ such that 
\[
\dist(0,z_k) = \len(\gamma_k) \leq \dist(0,H_{R_k}) + \delta R_k \leq R_k (\eta(u) -\delta).
\]
This however implies that 
\[
\limsup_{k\to\infty} \frac{\dist(0,z_k)}{|z_k|}\leq 
\limsup_{k\to\infty}  \frac{R_k(\eta(u)-\delta)}{|z_k|} \leq \eta(u)-\delta,
\]
contradicting the uniform convergence of the distance function on $S^{d-1}$ proved in Theorem \ref{thm:distance_in_the_plane}. Thus the event in \eqref{eq:lemma_on_point_hyperplane_distance} has probability zero and the claim follows. 
\end{proof}

Next, we show that the value of $s_{M,N}$ defined in \eqref{eq:LD_defn_of_s_M_N} cannot be much smaller than $M\eta(u)$.
	
\begin{lemma}\label{lem:point_to_hyperplane_lemma_2}
For every $\varepsilon>0$ 
	\begin{equation}
		\lim_{M\to\infty}\max_{N\leq M} \prob_u\brk{s_{M,N} \leq 	M(\eta(u)-\varepsilon)}=0.
	\end{equation}
\end{lemma}

\begin{proof}
If $\eta(u)=0$, then there is nothing to prove since $s_{M,N}\geq 0$. We therefore assume that $\eta(u)>0$. Let
\begin{equation}\label{eq:defn_of_r_M_N}
	r_{M,N} = \inf\left\{\len(\gamma) ~:~ \gamma \text{ is a path from } \{0\}\times [0,N]^{d-1} \text{ to } H_M\right\}.
\end{equation}
Since $r_{M,N}$ is an infimum of path lengths over a set larger than that defining $s_{M,N}$, it follows that $r_{M,N}\leq s_{M,N}$, therefore 
\[
\prob_u(s_{M,N} \leq M(\eta(u)-\varepsilon))\leq \prob_u(r_{M,N} \leq M(\eta(u)-\varepsilon))
\le \prob_u(r_{M,M} \leq M(\eta(u)-\varepsilon)),
\]
where in the last passage we used the fact that $r_{M,N}$ is decreasing in $N$.
Hence it is enough to prove that 
\[
\lim_{M\to\infty} \prob_u(r_{M,M} \leq M(\eta(u)-\varepsilon))=0.
\] 

We now show that, in fact, it is sufficient to prove that for some fixed choice of  $\delta=\delta(\varepsilon)\in(0,1/2)$ ,
\begin{equation}\label{eq:point_to_hyperplane_lemma_2_eq_1}
		\lim_{M\to\infty} \prob_u(r_{M,2\delta M} \leq M(\eta(u)-\varepsilon))=0.	
\end{equation}

Indeed, let $\delta\in(0,1/2)$ and for $v\in H_0$ define $r_M(v) = \dist(v,H_M)$. Since $\{0\}\times [0,M]^{d-1}$ is contained in the union of  boxes $\{\{0\}\times (2\delta Mk + [0,2\delta M]^{d-1})\}_{k\in A(\delta)}$ with 
\[
	A(\delta) = \left\{k=(k_2,\ldots,k_d)\in\bbZ^{d-1} :~ 0\leq k_i \leq \lceil 1/(2\delta) \rceil,\quad  \forall 2\leq i\leq d\right\}
\]
it follows that 
\[
\begin{aligned}
	r_{M,M}  & = \inf\{r_M(v) ~:~ v\in \{0\}\times [0,M]^{d-1}\} \\ 
	 & \geq \min_{k\in A(\delta)} \inf\{r_M(v) ~:~ v\in \{0\} \times (2\delta Mk + [0,2\delta M]^{d-1})\}.
\end{aligned}
\]
Since the set $A(\delta)$ is finite, and since each infimum for a fixed $k$ has the same distribution as $r_{M,2\delta M}$ it is indeed sufficient to prove \eqref{eq:point_to_hyperplane_lemma_2_eq_1}. 

Finally, let us prove \eqref{eq:point_to_hyperplane_lemma_2_eq_1}. For every $v\in H_0$ such that $|v|_\infty \leq 2\delta M$ we have $\dist(0,v) \leq 2\delta M$, and therefore 
\[
	r_M(0) \leq r_M(v) + 2\delta M. 
\]
Taking the infimum over all such $v$'s yields
\[
	r_M(0) \leq r_{M,2\delta M} + 2\delta M,
\]
and therefore for $\delta<\varepsilon/4$ ,
\[
\begin{aligned}
		\prob_u(r_{M,2\delta M} \leq M(\eta(u)-\varepsilon)) & \leq 
		 \prob_u(r_M(0)-2\delta M \leq M(\eta(u)-\varepsilon))\\
		 & = \prob_u(\dist(0,H_M)-2\delta M \leq M(\eta(u)-\varepsilon))\\
		 & \leq \prob_u(\dist(0,H_M) \leq M(\eta(u)-\varepsilon/2))
\end{aligned}
\]
which, by Lemma \ref{lem:point_to_hyperplane_lemma_1}, tends to zero as $M$ tends to infinity.
\end{proof}

\begin{proof1}{Proof of Theorem \ref{thm:Chemical_distance} (lower bound)}
For $\eta(u)=0$ the claim is trivial since $\prob_u(\dist(0,x)<0)=0$. 
For $\eta(u)>0$, let $N = \min\left\{M,\left\lfloor \frac{M\varepsilon}{4\eta(u)}\right\rfloor \right\}$. For $Q\geq \frac{R}{M+N}-1$ and $\beta> 0$, 
\begin{equation}\label{eq:proof_of_chemical_distance_1}
\begin{split}
\prob_u\brk{\sum_{q=0}^{Q-1} X_q(M,N) < R(\eta(u)-\varepsilon)} 
&\leq e^{\beta  R(\eta(u)-\varepsilon)}\cdot \Exp_u\Brk{\exp\brk{-\beta  \sum_{q=0}^{Q-1} X_q(M,N)}} \\
&\hspace{-4cm}=  e^{\beta  R(\eta(u)-\varepsilon)}\cdot \Exp_u\Brk{e^{-\beta X_1(M,N)}}^Q \\
&\hspace{-4cm}\leq  e^{\beta  (Q+1)(M+N)(\eta(u)-\varepsilon)}\cdot \Exp_u\Brk{e^{-\beta X_1(M,N)}}^Q \\
&\hspace{-4cm}=  e^{\beta  (M+N)(\eta(u)-\varepsilon)}\cdot \brk{ e^{\beta(M+N)(\eta(u)-\varepsilon)}\cdot  \Exp_u\Brk{e^{-\beta X_1(M,N)}}}^Q \\
&\hspace{-4cm}\leq  e^{\beta  (M+N)(\eta(u)-\varepsilon)}\cdot \brk{ e^{\beta(M+N)(\eta(u)-\varepsilon)}\cdot  \brk{e^{-\beta M(\eta(u)-\varepsilon/2)}+\prob_u(X_1(M,N) < M(\eta(u)-\varepsilon/2))}}^Q \\
&\hspace{-4cm}\leq  e^{\beta  (M+N)(\eta(u)-\varepsilon)}\cdot \brk{ e^{-\beta M\varepsilon/4} +  e^{2\beta M (\eta(u)-\varepsilon)}\prob_u(X_1(M,N) < M(\eta(u)-\varepsilon/2))}^Q.
\end{split}
\end{equation}

On the first line we used Markov's inequality; in the passage to the second line we used the fact that the $X_q(M,N)$ are i.i.d.; in the passage to the third line we used the fact that $R\le (Q+1)(M+N)$; the passage to the fourth line is an immediate algebraic identity; the passage to the fifth line follows from the inequality $\Exp_u[e^{-\beta X}] \le e^{-\beta a} + \prob_u(X< a)$, valid for every positive random variable $X$; finally, the passage to the sixth line follows from the choice of $N$.

Recalling that $X_1(M,N)\sim s_{M,N}$, we obtain from Lemma \ref{lem:point_to_hyperplane_lemma_2} that $0<\prob_u(X_1(M,N) < M(\eta(u)-\varepsilon/2))<1$ for large enough values of $M$. For every such value of $M$ one can find $\beta_M$, depending only on $M,d,\varepsilon,u$, such that
\[
	(\prob_u(X_1(M,N) < M(\eta(u)-\varepsilon/2)))^{-1/3} \leq e^{2\beta_M M(\eta(u)-\varepsilon)} \leq (\prob_u(X_1(M,N) < M(\eta(u)-\varepsilon/2)))^{-1/2}.
\]
When combined with \eqref{eq:proof_of_chemical_distance_1} this implies 
\begin{equation}\label{eq:proof_of_chemical_distance_2}
\begin{split}
\prob_u\brk{\sum_{q=0}^{Q-1} X_q(M,N) < R(\eta(u)-\varepsilon)} 
&\leq  e^{\beta_M(M+N)(\eta(u)-\varepsilon)}\\
&\hspace{-4cm}\cdot \bigg( (\prob_u(X_1(M,N) < M(\eta(u)-\varepsilon/2)))^{\frac{3\varepsilon}{24(\eta(u)-\varepsilon)}} + (\prob_u(X_1(M,N) < M(\eta(u)-\varepsilon/2)))^{1/2}\bigg)^Q.
\end{split}
\end{equation}

Using Lemma \ref{lem:point_to_hyperplane_lemma_2} one more time, we can choose $M=M(\varepsilon,d,u)$ large enough so that  
\begin{equation}\label{eq:proof_of_chemical_distance_3}
\begin{split}
& (\prob_u(X_1(M,N) < M(\eta(u)-\varepsilon/2)))^{\frac{3\varepsilon}{24\eta(u)-\varepsilon)}} + (\prob_u(X_1(M,N) < M(\eta(u)-\varepsilon/2)))^{1/2}\\
&\qquad \leq \brk{32d \cdot \max\left\{2,\frac{ 8\eta(u)}{\varepsilon}\right\}}^{-d}
	\leq \brk{32d\frac{M}{N} }^{-d}.
\end{split}
\end{equation}

For such choices of $M$ and $\beta$, we get from \eqref{eq:proof_of_chemical_distance_2}, \eqref{eq:proof_of_chemical_distance_3} and Proposition \ref{prop:LD_key_prop}   
\[
\begin{split}
\prob_u\brk{\dist(0,Re_1) < R(\eta(u)-\varepsilon)}  
&  \leq 	e^{\beta(M+N)(\eta(u)-\varepsilon)}\cdot \sum_{Q \geq \frac{R}{M+N}-1} \brk{2d\brk{16\frac{M}{N}}^d}^Q \brk{32d\frac{M}{N} }^{-dQ} \\
&\leq  e^{\beta(M+N)(\eta(u)-\varepsilon)}\cdot \sum_{Q \geq \frac{R}{M+N}-1}2^{-dQ}\\
& \leq e^{\beta(M+N)(\eta(u)-\varepsilon)}\cdot 2^{-\frac{dR}{2M}+d+1}.
\end{split}
\]
Recalling that $M$, $N$ and $\beta$ are fixed, this gives the desired exponential decay in $R$. 

Given the result for $x=Re_1$ we turn to deal with general points, $x\in\R^d$, $|x|=R$. Due to the invariance of $\prob_u$ under rotations, $\dist(0,x)$ for $|x|=R$ has the same distribution as $\dist(0,Re_1)$, and therefore for large enough $R$
\[
	\prob_u(\dist(0,x) < (\eta(u)-\varepsilon/2)R) < e^{-c_1R},\quad \forall x\in \R^d,\text{ such that } |x|=R.
\]

Taking an $\varepsilon/2$-net $\mathcal{N}$ on $S^{d-1}$ such that $|\mathcal{N}|\leq \frac{C(d)}{\varepsilon^d}$ we get that 
\[
	\prob_u(\exists x\in \mathcal{N} \text{ such that } \dist(0,Rx) < (\eta(u)-\varepsilon/2)R) \leq \frac{C(d)}{\varepsilon^d}e^{-c_1R}.
\] 
For every $x\in \R^d$ such that $|x|=R$ there exists a $y\in\mathcal{N}$, such that $|x/R-y|<\varepsilon/2$, and therefore 
\[
	\dist(0,Ry) \leq \dist(0,x)+ |Ry-x| \leq \dist(0,x)+\frac{\varepsilon}{2}R.
\]
Hence, 
\[
\begin{split}
& \prob_u\brk{\exists x\in\R^d~:~|x|=R,~\dist(0,x)<(\eta(u)-\varepsilon) R} \\
&\qquad\leq  	\prob_u\brk{\exists y\in \mathcal{N} \text{ such that } \dist(0,Ry) < (\eta(u)-\varepsilon/2)R} \\
&\qquad \leq \frac{C}{\varepsilon^d}e^{-c_1R},
\end{split}
\]
which concludes the proof. 
\end{proof1}

\bigskip\bigskip

Next, we turn to prove the upper bound in Theorem \ref{thm:Chemical_distance}, which states  that for every $u\in[0,u_*)$ and every $\e>0$ there exists a constant $c_1$, depending only on $d$, $u$ and $\varepsilon$, such that for $R$ large enough
\begin{equation}
	\prob_u\brk{\exists x\in\R^d~:~|x|=R,~\dist(0,x)> (\eta(u)+\varepsilon) R}< e^{-c_1 R}.
\end{equation}

We start with the following lemma.

\begin{lemma} \label{lem:LDP_upper_bound_lemma}
For $S<R$ let 
\[
p_{S,R} = \inf\left\{\len(\gamma) ~:~ \begin{array}{l} \gamma \text{ is a path from } Se_1 \text{ to } Re_1 \\
	\text{such that } H_{S}< \gamma < H_{R}
\end{array}
\right\}.
\]
Then 
\[
	\lim_{R\to\infty}\frac{\Exp_u[p_{0,R}]}{R} = \eta(u).
\]
\end{lemma}

\begin{proof}
We follow \cite[Theorem 4.3.7]{HW65}. Since $\dist(0,Re_1) \leq p_{0,R}$,
\begin{equation}\label{eq:LDP_upper_bound_lemma_one_side}
	 \liminf_{R\to\infty}\frac{\Exp_u[p_{0,R}]}{R} \geq \lim_{R\to\infty}\frac{\Exp_u[\dist(0,Re_1)]}{R} = \eta(u). 
\end{equation}
For $k\geq 0$ define 
\[
p^k_{S,R} = \inf\left\{\len(\gamma) ~:~ \begin{array}{l} \gamma \text{ is a path from } Se_1 \text{ to } Re_1 \\
	\text{such that } H_{S-k}< \gamma < H_{R+k}
	\end{array}
	\right\},
\]
so in particular, $p_{S,R}=p^0_{S,R}$. Since for every $k\geq 0$ and $0<R_1<R_2$, 
\begin{equation}\label{eq:LDP_upper_bound_lemma_subadditivity}
	p^k_{0,R_1+R_2} \leq p^k_{0,R_1} + p^k_{R_1,R_1+R_2}
\end{equation}
one can apply the subadditive ergodic theorem \cite{Ki73} for $p^k_{0,R}$ to obtain
\[
	\lim_{R\to\infty}\frac{\Exp_u[p^k_{0,R}]}{R} = \eta^k(u)
\]
which, by the definition of $p^k_{0,R}$, satisfies 
\begin{equation}\label{eq:lemma_for_LDP_upper_bound}
	\eta(u) \leq \eta^k(u) \leq \eta^{k-1}(u) \leq \eta^0(u),\quad \forall k\geq 1.
\end{equation}
Noting that for every $R,k>0$
\[
	p_{-k,R+k} \leq p_{-k,0}  + p^k_{0,R} + p_{R,R+k},
\]
it follows by fixing $k$, taking expectation, dividing by $R$ and taking the limit $R\to\infty$ that 
\[
	\eta^0(u) \leq \eta^k(u),
\]
which together with \eqref{eq:lemma_for_LDP_upper_bound} implies that $\eta^0(u)=\eta^k(u)$ for every fixed $k>0$. Since  for every fixed $\omega\in\Omega$ and $R>0$, $p_{0,R}^k$ is a monotonically decreasing function in $k$  converging to $\dist(0,Re_1)$, it follows from the monotone convergence theorem that 
\begin{equation}\label{eq:lemma_for_LDP_upper_bound_1}
	\lim_{k\to\infty}\Exp_u[p^k_{0,R}] = \Exp_u[\dist(0,Re_1)].
\end{equation}
Using once again the sub-additivity \eqref{eq:LDP_upper_bound_lemma_subadditivity} we conclude that 
\[
\frac{\Exp_u[p^k_{0,R}]}{R} \geq \eta^k(u)=\eta^0(u)
\] 
for every $R>0$ by Fekete's subadditive lemma. Combined with \eqref{eq:lemma_for_LDP_upper_bound_1} this yields 
\[
	\frac{\Exp_u[\dist(0,Re_1)]}{R} \geq \eta^0(u).
\]
Letting $R\to\infty$ implies $\eta^0(u) \leq \eta(u)$.  Together with \eqref{eq:LDP_upper_bound_lemma_one_side}, this completes the proof. 
\end{proof}

\begin{proof1}{Proof of Theorem \ref{thm:Chemical_distance} (upper bound)}
Proving the upper bound is in fact much simpler than proving the lower bound. 
Fix $\varepsilon>0$. By Lemma \ref{lem:LDP_upper_bound_lemma} there exists \revision{a sufficiently large} $R_0>2$ such that 
\begin{equation}\label{eq:LDP_upper_bound_1}
	\frac{\Exp[p^0_{0,R_0}]}{R_0}<\eta(u)+\frac{\varepsilon}{5}.
\end{equation}

For $i\geq 0$ let $X_i = p_{iR_0,(i+1)R_0}$. The random variables $(X_i)_{i\geq 0}$ are i.i.d. (with the same distribution as $p_{0,R_0}$), and 
\begin{equation}\label{eq:LDP_upper_bound_2}
	\dist(0,Re_1) \leq \sum_{i=0}^{\lfloor R/R_0\rfloor -1} X_i + R_0.
\end{equation}
Using \eqref{eq:LDP_upper_bound_1} and \eqref{eq:LDP_upper_bound_2} we deduce that for  every $R > 5R_0/\e$,
\[
\begin{split}
	\prob_u\brk{\dist(0,Re_1) > (\eta(u)+\varepsilon)R} & \leq \prob_u\brk{\sum_{i=0}^{\lfloor R/R_0\rfloor -1} X_i > (\eta(u) + \varepsilon-R_0/R)R}\\
	& \leq \prob_u\brk{\sum_{i=0}^{\lfloor R/R_0\rfloor -1} (X_i-\Exp_u[X_i]) > \brk{\eta(u) + \varepsilon-\frac{R_0}{R} - \frac{\Exp_u[X_1]}{R_0}  }R}\\
	& \leq \prob_u\brk{\sum_{i=0}^{\lfloor R/R_0\rfloor -1} (X_i-\Exp_u[X_i]) > \frac{3}{5}\varepsilon R}.
\end{split}
\]
Applying the function $x\mapsto e^{\beta x}$ (for some $\beta>0$) to both sides in the last term, using the Markov inequality and then the independence of the $X_i$'s, we can bound the last term on the right-hand side by
\begin{equation}\label{eq:LDP_upper_bound_3}
\begin{split}
	e^{-\frac{3\beta\varepsilon}{5} R} \cdot \Exp_u\Brk{e^{\beta\brk{\sum_{i=0}^{\lfloor R/R_0\rfloor -1} (X_i-\Exp_u[X_i])}}}
	& = e^{-\frac{3\beta\varepsilon}{5} R} \cdot \Exp_u\Brk{e^{\beta (p_{0,R_0}-\Exp_u[p_{0,R_0}])}}^{\left\lfloor \frac{R}{R_0}\right\rfloor}\\
	& \leq  e^{-\frac{3\beta\varepsilon}{5} R} \cdot \Exp_u\Brk{e^{ \beta (p_{0,R_0}-\Exp_u[p_{0,R_0}])}}^{\frac{R}{R_0}}.
\end{split}
\end{equation}
Since $\Exp_u\Brk{e^{ \beta (p_{0,R_0}-\Exp_u[p_{0,R_0}])}} \leq \Exp_u\Brk{e^{ \beta p_{0,R_0}}} \leq e^{\beta R_0} <\infty$ we can expand the exponential inside the expectation into a power series in $\beta$ and obtain that 
\[
	\Exp_u\Brk{e^{ \beta (p_{0,R_0}-\Exp_u[p_{0,R_0}])}} = 1+ O(\beta^2).
\]
Hence \eqref{eq:LDP_upper_bound_3} is bounded by $	\exp\brk{\brk{\frac{O(\beta^2)}{R_0} - \frac{3\beta\varepsilon}{5}}R}$. By taking $\beta=\beta(\varepsilon)>0$ small enough we can make the last term exponentially decaying in $R$, thus completing the proof. 
\end{proof1}

\subsection{Proof of Theorem \ref{thm:too_long_Euclidean_distance_is_very_unlikely}}

The proof of Theorem 5.2 also follows ideas of Kesten \cite{Ke84}. Unlike the proof of Theorem 5.1, some additional elements are needed in order to apply to the continuous case. We start with some preliminary results.

\begin{lemma}
\label{lem:num_of_crossing_of_a_finite_box_by_a_true_geodesic}
There exists a constant $C_4 = C_4(d)$ such that for every \revision{$M>1$}, the Euclidean length of every true geodesic contained in a box of side length $M$
is at most $C_4 M^{\revision{d}}$.
\end{lemma}

\begin{proof}
\revision{
Denote the box by $B_M$.
Let $\omega\in\Omega$ be a configuration, and let $\gamma$ be a true geodesic contained in $B_M$.
Since $\len(\gamma)\le \diam_\euc(B_M) = \sqrt{d}M$, the Euclidean length that $\gamma$ acquires in $B_M\setminus \calS(\omega)$ is at most $\sqrt{d}M$.
Therefore, it suffices to show that the Euclidean length that $\gamma$ acquires in $\calS(\omega) \cap B_M$ is $O(M^{d})$.

Let $\{A_i\}_{i=1}^N$ be the connected components of $\calS(\omega)\cap B_M$ which $\gamma$ intersects.
Note that $\gamma\cap A_i$ is connected (otherwise it would not be a true geodesic), so there are well-defined entry and exit points to $\gamma\cap A_i$.
Denote by $x_i^I\in A_i$ (resp.~$x_i^L\in A_i$) the point in $\supp(\omega)$, from which the entry point (resp.~exit point) of $\gamma\cap A_i$ is of distance at most $1$.
Let $x_i^I=x_i^1,x_i^2,\ldots, x_i^{k_i}=x_i^L$ be points in $\supp(\omega)\cap A_i$, such that $|x_i^j-x_i^k|\le 2$ if $|k-j|\le 1$, $|x_i^j-x_i^k|>2$ if $|k-j|>1$. 
We can always find such points by taking $\supp(\omega)\cap A_i$ and omitting points.
By construction, $x_i^j\in A_i$ for every $j$, and therefore the length $\gamma$ acquires in $A_i$ is at most $2k_i$ (the length of connecting $x_i^j$ with $x_i^{j+1}$ plus the length of connecting the entry and exit points with $x_i^I$ and $x_i^L$).

The length $\gamma$ acquires in $\calS(\omega)\cap B_M$ is therefore $2k_1+\ldots+2k_N$.
However, by construction, the unit balls centered at $\{x_i^{2j}\}_{i\le N, 2j\le k_i}$ are mutually disjoint. 
By a volume consideration, the number of disjoint unit balls in $B_M$ is at most $\kappa_d M^d$, where $\kappa_d$ is the volume of the Euclidean unit ball in $\R^d$.
Therefore $k_1+\ldots+k_N\le 2\kappa_d M^d$, and so the total length $\gamma$ acquires in $\calS(\omega)\cap B_M$ is bounded by $4\kappa_d M^d$.
}
%
\end{proof}

We also recall the following variant of a result by Roy \cite{Ro90}; see also \cite[Lemma 3.3]{MR96}:

\begin{theorem}[\cite{Ro90}]
\label{thm:diam_of_clusters} 
For $x\in\R^d$ denote 
\[
\calS(x;\omega) = \left\{\begin{array}{ll}
	\text{the connected component of }x \text{ in } \calS(\omega), & \quad x\in \calS(\omega)\\
	\emptyset, & \quad x\notin \calS(\omega)
\end{array}\right.
\]
Then, there exist for every $u<u_*$ positive constants $C,c$ depending only on $u$ and $d$ such that 
\[
\prob_u(\diam_\euc(\calS(0;\omega))>t)\leq Ce^{-ct}.
\]
\end{theorem}

In order to prove Theorem \ref{thm:too_long_Euclidean_distance_is_very_unlikely}, as well as for future use, we will need a stronger version of Theorem \ref{thm:diam_of_clusters}:
\begin{lemma}
\label{lem:diam_of_clusters_2}
Let
\[
W = \bigcup_{x\in [0,1]^d} \calS(x;\omega). 
\]
Then, for every $u<u_*$ there exist positive constants $C_5,c_6$ depending only on $u$ and $d$, such that 
\[
\prob_u(\diam_\euc(W)>t)\leq C_5e^{-c_6t}.
\]	
\end{lemma}

\revision{
\begin{proof}
Observe that every unit ball $\overline{B(x,1)}$ for $x\in\R^d$ contains a point in the grid $\frac{1}{2}\bbZ^d$. Therefore, each connected component in $W$ must contain a point from the set
\[	
	L=\left\{q\in \frac{1}{2}\bbZ^d ~:~ |q|_2 <2\right\}
\]
whose size is finite (and depends only on $d$). 

Assume that $\diam_\euc(W)>t$. Then, there exists a connected component whose Euclidean diameter is at least $(t/2-\sqrt{d})$. By Theorem \ref{thm:diam_of_clusters}, 
\[
\prob_u(\diam_\euc(W)>t)\leq \sum_{x\in L}\prob_u(\diam_\euc(\calS(x)) >t/2-\sqrt{d})\leq Ce^{-ct}.
\]
\end{proof}
}

%

\begin{proof1}{Proof of Theorem \ref{thm:too_long_Euclidean_distance_is_very_unlikely}} 
The proof follows by an argument very similar to the one used in the proof of Theorem \ref{thm:Chemical_distance}. We start with Part (1), showing that there exist constants $\alpha$, $C_2$ and $c_3$, such that for every $R>0$,
\[
\prob_u\brk{\begin{array}{c} \exists \text{ a true geodesic path starting at } 0\\ \text{ such that } \len_\euc(\gamma)\geq R \text{ and } \len(\gamma)<\frac{1}{\alpha} \len_\euc(\gamma) \end{array}}\leq C_2e^{-c_3R}.
\]

First observe that 
\[
\begin{aligned}
& \prob_u\brk{\begin{array}{c} \exists \text{ a true geodesic path starting at } 0\\ \text{ such that } \len_\euc(\gamma)\geq R \text{ and } \len(\gamma)<\frac{1}{\alpha} \len_\euc(\gamma) \end{array}}\\
\leq & \sum_{S=\lfloor R\rfloor}^\infty \prob_u\brk{\begin{array}{c} \exists \text{ a true geodesic path starting at } 0\\ \text{ such that } \len_\euc(\gamma)\in [S,S+1] \text{ and } \len(\gamma)<\frac{1}{\alpha} (S+1) \end{array}}
\end{aligned}
\]
so it is enough to show that there exist constants $\alpha,C,c_3$ such that for every natural number $S$
\[
\prob_u\brk{\begin{array}{c} \exists \text{ a true geodesic path starting at } 0\\ \text{ such that } \len_\euc(\gamma)\in [S,S+1] \text{ and } \len(\gamma)<\frac{1}{\alpha} (S+1) \end{array}}\leq Ce^{-c_3S}.
\]

Fix $S\in\bbN$ and let $\gamma$ be a true geodesic path starting at the origin such that $\len_\euc(\gamma)\in [S,S+1]$. Fix $M\in\bbN$ and $N=1$ and define the sequences $t_q$ and $x_q$ , $q=1,\dots,Q$ in the same way as in the proof of Proposition \ref{prop:LD_key_prop}. 
By construction, each of the segments $\gamma|_{[t_q,t_{q+1}]}$ is contained in a box of side length $M+1$. By Lemma \ref{lem:num_of_crossing_of_a_finite_box_by_a_true_geodesic} the Euclidean length of a true geodesic in each of the boxes is at most $C_4(M+1)^\revision{d}$, hence
\[
Q \geq \frac{S}{C_4 (M+1)^\revision{d}}-1.
\] 

By repeating the argument of Proposition \ref{prop:LD_key_prop} we get that for every $\beta \geq 0$
\[
\begin{split}
& \prob_u\brk{\begin{array}{c} \exists \text{ a true geodesic path starting from } 0\\ \text{ such that } \len_\euc(\gamma)\in[S,S+1] \text{ and } \len(\gamma)<\frac{1}{\alpha} (S+1) \end{array}} \\
&\qquad \leq  \sum_{Q \geq \frac{S}{C_4(M+1)^\revision{d}}-1} \brk{2d\brk{16M}^d}^Q \prob_u\brk{\sum_{q=0}^{Q-1} X_q(M,1) <\frac{1}{\alpha} (S+1)}	\\
&\qquad \leq  \sum_{Q \geq \frac{S}{C_4(M+1)^\revision{d}}-1} \brk{2d\brk{16M}^d}^Q e^{ \beta (S+1)/\alpha}\Exp_u\Brk{\prod_{q=0}^{Q-1} e^{-\beta X_q(M,1)}}\\
&\qquad =  e^{ \beta (S+1)/\alpha}\cdot \sum_{Q \geq \frac{S}{C_4(M+1)^\revision{d}}-1} \brk{2d\brk{16M}^d}^Q \Exp_u\Brk{ e^{-\beta X_1(M,1)}}^Q.\\
\end{split}
\]
With $W$ as in Lemma \ref{lem:diam_of_clusters_2}, 
\[
	\prob_u\brk{\diam_\euc(W)>\frac{M}{2}}\leq C_5e^{-c_6M/2}.
\]
Each connected component of $\calS$ is composed $\prob_u$-a.s.~of finitely many balls; see Lemma \ref{lem:Basic_properties_of_M_R}. Hence  connected components of $\calS$, as well as $W$, are compact sets. Therefore, one can find \revision{a small enough} $\delta=\delta(M)>0$ such that 
\[
	\prob_u(\dist(W,\calS\setminus W)\geq \delta(M)) \geq 1-C_5e^{-c_6M/2}.
\]
It follows that with $\prob_u$-probability at least $1-2C_5e^{-c_6M/2}$ the set $W$ is contained within the box $\Brk{-M/2,M/2 \revision{+1}}^d$ and is at a distance at least $\delta(M)$ from any other cluster of $\calS$, which in particular implies that $e^{-\beta X_1(M,1)} \leq e^{-\beta\delta}$.
Consequently, 
\[
\Exp_u[e^{-\beta X_1(M,1)}] \leq e^{-\beta \delta(M)} + 2C_5e^{-c_6M/2}
\]
and thus 
\[
\begin{split}
& \prob_u\brk{\begin{array}{c} \exists \text{ a true geodesic path starting at } 0\\ \text{ such that } \len_\euc(\gamma)\in[S,S+1] \text{ and } \len(\gamma)<\frac{1}{\alpha} (S+1) \end{array}} \\
&\qquad \leq  e^{\beta (S+1)/\alpha} \sum_{Q\geq \frac{S}{C_4(M+1)^{\revision{d}}}-1} \brk{2d(16M)^d}^Q  \brk{e^{-\beta \delta(M)} + 2C_5e^{-c_6M/2}}^Q.
\end{split}
\]
Taking $M$ large enough so that $2d(16M)^d \cdot 4C_5 e^{-c_6M/2}< \frac{1}{4}$, and then $\beta$ large enough so that $e^{-\beta \delta(M)} \leq 2C_5e^{-c_6M/2}$, and finally $\alpha$ large enough so that $e^{ \beta/\alpha } < 2^{\frac{1}{C_4(M+1)^\revision{d}}}$ gives 
\[
\prob_u\brk{\begin{array}{c} \exists \text{ a true geodesic path starting from } 0\\ \text{ such that } \len_\euc(\gamma)\in[S,S+1] \text{ and } \len(\gamma)<\frac{1}{\alpha} (S+1) \end{array}} \leq e^{\beta/\alpha}\brk{\frac{1}{2}}^{\frac{\revision{S}}{C_4(M+1)^\revision{d}}},
\]
proving Part (1) of Theorem \ref{thm:too_long_Euclidean_distance_is_very_unlikely}.

We next prove Part (2), namely that for every $x,y\in \R^d$ and for every $R>\alpha|x-y|$
\[
\prob_u\brk{\exists \gamma \in \Gamma_0(x,y)~:~\len_\euc(\gamma)> R }\leq C_2e^{-c_3R}.
\]

By the invariance of the measure $\prob_u$ under translations and rotations, it is sufficient to prove the result for the case $x=0$, $y=Se_1$ for some $S>0$. If $\gamma\in\Gamma_0(0,Se_1)$ is a path such that $\len_\euc(\gamma)>R>\alpha S$, then we can find times $t_q$ and points $x_q$, $q=0,\dots,Q$, as in Proposition \ref{prop:LD_key_prop} (with $N=1$), with $Q>{\alpha S}(M+1)-1$. Since $\len(\gamma) =\dist(0,Se_1) \leq S$ the result follows by the same argument used to prove Part (1). 
\end{proof1}

\section{Further result for uniform point distributions}
\label{sec:concentration_results}

In this section we exploit the results obtained in the previous two sections to prove more results for  uniform point distributions. These include the geometric concentration of geodesics and further properties of the function $\eta(u)$. 

\subsection{Geometry of geodesics}

In order to obtain results on the limiting distance in the setting of general intensity measures, we will need control over the geometry of geodesics, or approximate geodesics,  connecting pairs of points. For $x,y\in\R^d$, we call a path $\gamma$ between $x$ and $y$ an $\varepsilon$-geodesic if 
\beq
	\len(\gamma)-\dist(x,y) \leq \varepsilon|x-y|.
\eeq
We denote by $\Gamma_\varepsilon(x,y)$ the set of $\varepsilon$-geodesics which are also Euclidean geodesics inside $\calS$. In particular, $\Gamma_0(x,y)$ is the set of true geodesics between $x$ and $y$ defined in Section \ref{sec:large_deviations}.  

The goal of this subsection is to show that there exist geodesics that do not deviate significantly from Euclidean segments:

\begin{proposition}
\label{pn:concentration_est}
Let $u\in [0,u_*)$. There exist $C_7,c_8>0$ such that for every $\varepsilon>0$ and sufficiently large (depending on $\e$) $|x-y|$, 
\beq\label{eq:prop_distanc_of_almost_geodesic_from_segment}
	\prob_u\brk{\forall \gamma \in \Gamma_\varepsilon(x,y) ~:~ d_H(\gamma,[x,y])>\varepsilon|x-y|} \leq C_7 e^{-c_{8}|x-y|^{1/2}},
\eeq
where $[x,y]$ is the linear segment connecting $x$ and $y$. 
\end{proposition}

\begin{proof}
We follow the proof of  \cite[Proposition 3.2]{BLPR15}. Fix $\e>0$ and $x,y\in\R^d$. We will show that if $|x-y|$ is large enough, then there exists with $\prob_u$-probability $\geq 1-C_7e^{-c_{8}|x-y|}$ a curve $\gamma\in\Gamma_\e(x,y)$ satisfying 
\[
d_H(\gamma,[x,y])\leq\varepsilon|x-y|.
\]

Let \revision{$N=\lceil 20\alpha/\varepsilon\rceil$}, with $\alpha$ as in Theorem \ref{thm:too_long_Euclidean_distance_is_very_unlikely}. Define the sequence of vertices $z_k$, $k=0,\dots,N$, along $[x,y]$ by 
\[
	z_k = \brk{1-\frac{k}{N}} x + \frac{k}{N} y,\quad \forall 0\leq k\leq N.
\] 
For $0\leq k\leq N-1$, let $\gamma_k=\gamma_k(\omega)\in \Gamma_0(z_k,z_{k+1})$ be a true geodesic connecting $z_k$ and $z_{k+1}$ and define $\gamma = (\ldots (\gamma_0 * \gamma_1) * \ldots *) *\gamma_{N-1})$ to be their concatenation, connecting $x$ and $y$. 

Since $|z_k-z_{k+1}| = |x-y| / N$ is of order $|x-y|$ it follows from Theorem \ref{thm:Chemical_distance} that once $|x-y|$ is large enough
\begin{equation}\label{eq:geometric_concentration_1}
	|\dist(z_k,z_{k+1}) - \eta(u)\cdot |z_k-z_{k+1}| | \leq \frac{\varepsilon}{2}|z_k - z_{k+1}|, \quad \forall 0\leq k \leq N-1,
\end{equation}
with $\prob_u$-probability at least $1- Ne^{-c_1|x-y|}$. 
Similarly, whenever $|x-y|$ is large enough we have with $\prob_u$-probability at least $1-e^{-c_1|x-y|}$ that 
\begin{equation}\label{eq:geometric_concentration_2}
	|\dist(x,y)- \eta(u)\cdot |x-y|| \leq \frac{\varepsilon}{2} |x-y|.
\end{equation}
Consequently, under the events in \eqref{eq:geometric_concentration_1} and \eqref{eq:geometric_concentration_2},
\begin{equation}
\begin{aligned}
	\len(\gamma)  & \leq \sum_{k=0}^{N-1} \len(\gamma_k)  = \sum_{k=0}^{N-1} \dist(z_k,z_{k+1}) \\ 
 & \leq  \sum_{k=0}^{N-1} \brk{\eta(u)\cdot |z_k-z_{k+1}| + \frac{\varepsilon}{2}|z_k-z_{k+1}|}\\
 & = \eta(u)\cdot |x-y|  + \frac{\varepsilon}{2} |x-y| \\ 
 & \leq \dist(x,y) + \varepsilon|x-y|,
\end{aligned}
\end{equation}
which implies that $\gamma \in \Gamma_\varepsilon(x,y)$. 
Finally, by Theorem \ref{thm:too_long_Euclidean_distance_is_very_unlikely}(2) we have with $\prob_u$-probability at least $1-C_2Ne^{-c_3|x-y|}\geq 1- C_7e^{-c_8|x-y|^{1/2}}$, assuming $|x-y|$ is sufficiently large.  
that 
\begin{equation}
	\len_\euc(\gamma_k) \leq \frac{\varepsilon}{10} |x-y| ,\quad \forall 0\leq k\leq N-1
\end{equation}
which implies that for every $0\leq k\leq N-1$, 
\begin{equation}
	\gamma_k \subset [z_k,z_{k+1}] + B\brk{0,\frac{\varepsilon}{5} |x-y|},
\end{equation}
and thus 
\begin{equation}
	\gamma \subset [x,y] +B\brk{0,\frac{\varepsilon}{5}|x-y|},
\end{equation} 
thus proving the existence of the required $\varepsilon$-geodesic.
\end{proof}


\subsection{An upper bound on $\eta(u)$}
\label{sec:eta_upper_bound}

The goal of this subsection is to prove an upper bound on the function $\eta(u)$:

\begin{proposition}
\label{prop:upper_bound_on_eta_u}
For every $u\in [0,u_*)$ 
\[
\eta(u)\leq e^{-u\kappa_d}.
\]
\end{proposition}

\begin{proof}
First observe that $x\in\calS(\omega)$ if and only if there is a point $y\in \supp(\omega)\cap \overline{B(x,1)}$. Since the the number of points in the support of $\omega$ inside $\overline{B(x,1)}$ is distributed like a Poisson random variable with parameter $u \kappa_d$ it follows that  
\[
	\prob_u(x\notin \calS(\omega)) = e^{-u\kappa_d}. 
\]

For $R>0$ let $\gamma_R:[0,1]\to\R^d$ denote the path $\gamma_R(t) = tRe_1$. Then
\[
\frac{\dist(0,Re_1)}{R}\leq \frac{\len(\gamma_R)}{R} = \frac{1}{R}\int_0^R \ind_{xe_1\notin \calS(\omega)}\Leb_1(dx).
\]
Taking expectation on both sides and using Fubini's theorem,
\[
\frac{\Exp_u[\dist(0,Re_1)]}{R} \leq \frac{1}{R}\int_0^R e^{-u\kappa_d} \Leb_1(dx) = e^{-u\kappa_d}.
\]
Taking the limit $R\to\infty$ and using Theorem \ref{thm:distance_in_the_plane} gives $\eta(u)\leq e^{-u\kappa_d}$. 
\end{proof}

\revision{
\begin{remark}
	In fact, one can show that the bound $e^{-u\kappa_d}$ is not tight, and that $\eta(u)<e^{-u\kappa_d}$ actually holds, by considering several paths instead of one as follows.
	Take large balls around $0$ and $Re_1$ (whose radii are independent of $R$).
	An $<e^{-u\kappa_d}$-bound is then obtained by considering several paths from $0$ to $Re_1$, which are at distance $>2$ from each other outside these balls. 
\end{remark}
}

\subsection{Continuity of $\eta(u)$}

\begin{proposition}
\label{prop:continuity_of_eta_u}
The function $\eta:[0,u_*)\to (0,1]$ is continuous. In addition, $\eta$ is monotonically decreasing in $u$ and $\eta(0)=1$. 
\end{proposition}

We start the proof by introducing a natural coupling of the probability measures $\prob_u$ for $u\geq 0$.
Let 
\[
\widehat{\Omega} = \BRK{\widehat{\omega}=\sum_{i\geq 0} \delta_{(x_i,u_i)} ~:~ \begin{array}{l}
x_i \in\R^d,~u_i\in [0,\infty) \text{ for all } i\geq 0 \text{ and } \widehat{\omega}(A\times [0,u])<\infty \\
\text{for all bounded, Borel-measurable } A\subset \R^d \text{ and } u\geq 0 
\end{array}}
\]
and for $u\geq 0$ and $\widehat{\omega}\in\widehat{\Omega}$ define 
\[
\widehat{\omega}_u = \sum_{{\scriptsize \begin{array}{c}(x_i,u_i)\in \text{supp}\widehat{\omega},\\u_i\leq u\end{array}}}\delta_{x_i}.
\]
Note that $\homega_u\in \Omega$, so $\dist(\cdot;\homega_u)$ is well-defined. Finally, let $\prob$ be the probability measure on $\widehat{\Omega}$ under which $\widehat{\omega}$ is distributed like a Poisson point process on $\R^d\times [0,\infty)$ with intensity $\Leb_d(dx)\times \Leb_1(dx)|_{[0,\infty)}$. 

One can verify that for every $u\geq 0$, the distribution of $\omega$ under $\prob_u$ is the same as the distribution of $\widehat{\omega}_u$ under $\prob$. That is, we constructed a coupling of the probability measures $\prob_u$ for $u\geq 0$ under which
\begin{equation}
\label{eq:coupling}
	\widehat{\omega}_u\leq \widehat{\omega}_{u'},\quad \forall 0\leq u\leq u'.
\end{equation}
We will denote by $\Exp$ expectation with respect to the probability measure $\prob$.

\begin{proof1}{Proof of Proposition \ref{prop:continuity_of_eta_u}}
Fix $u\in [0,u_*)$ and $\varepsilon>0$. Choose $\delta_0>0$ such that $[u-\delta_0,u+\delta_0]\subset [0,u_*)$ if $u>0$ and $[0,\delta_0]\subset [0,u_*)$ if $u=0$. We denote this compact interval by $I$.

Denote
\begin{equation}\label{eq:the_defn_of_the_event_A_S}
	A_{R} = \{\exists \gamma \in \Gamma_0(0,Re_1)~:~\len_\euc(\gamma)> \alpha R\}\cup\{|\dist(0,Re_1)-\eta(u)R|>\varepsilon R/6\},
\end{equation}
where $\alpha$ is as in Theorem \ref{thm:too_long_Euclidean_distance_is_very_unlikely}(2).

For every $u'\in I$, \revision{using the fact that $\eta(u),\eta(u')\in[0,1]$}, we have 
\begin{equation}\label{eq:cont_eq_1}
\begin{split}
& \left|\eta(u')-\eta(u)\right|  \revision{\le} 
\left|\Exp\left[(\eta(u')-\eta(u))\ind_{A_R^c}(\homega_{u'})\ind_{A_R^c}(\homega_{u})\right]\right| +\prob(A_{R}(\widehat{\omega}_{u'})) +   \prob(A_{R}(\widehat{\omega}_{u})) \\ 
&\qquad\le  \frac{1}{R}\left|\Exp\left[\left[\dist(0,Re_1;\widehat{\omega}_{u'})- \dist(0,Re_1;\widehat{\omega}_{u})\right]\cdot  \ind_{A^c_{R}}(\widehat{\omega}_{u'})\ind_{A^c_{R}}(\widehat{\omega}_{u})\right]\right|\\
&\qquad\qquad +\frac{\varepsilon}{3} +  \prob(A_{R}(\widehat{\omega}_{u'})) + \prob(A_{R}(\widehat{\omega}_{u})).
\end{split}
\end{equation}

By going back to the proofs of Theorem \ref{thm:Chemical_distance} and Theorem \ref{thm:too_long_Euclidean_distance_is_very_unlikely}(2) one can verify that both $\alpha$ and the constants $c_1,C_2,c_3$ can be chosen uniformly on the compact interval $I$. That is, there exist positive constants $\alpha$, $C_9$ and $c_{10}$ depending on $u$ and $\delta_0$ such that for every $u'\in I$, $\prob_{u'}(A_{R})\leq C_9 e^{-c_{10} R}$.

Due to the uniform bound on the probability of $A_{R}$ on the interval $I$ one can choose $R$ large enough (depending only on $\varepsilon>0$ and $I$) so that 
\begin{equation}\label{eq:cont_eq_2}
\prob(A_{R}(\widehat{\omega}_{u'}))+\prob(A_{R}(\widehat{\omega}_{u}))\leq \frac{\varepsilon}{3}.
\end{equation}
Combining \eqref{eq:cont_eq_1} and \eqref{eq:cont_eq_2}, it is enough to show the existence of $\delta>0$ such that for $|u-u'|<\delta$, 
\[
\frac{1}{R}\left|\Exp\left[\left[\dist(0,Re_1;\widehat{\omega}_{u'})- \dist(0,Re_1;\widehat{\omega}_{u})\right]\cdot  \ind_{A^c_{R}}(\widehat{\omega}_{u'})\ind_{A^c_{R}}(\widehat{\omega}_{u})\right]\right|<\frac{\varepsilon}{3}.
\]
This holds since on the event $\widehat{\omega}_{u'},\widehat{\omega}_{u}\in A^c_{R}$ the values of $\dist(0,Re_1;\widehat{\omega}_{u'})$ and $\dist(0,Re_1;\widehat{\omega}_{u})$ are the same whenever 
\[\supp(\widehat{\omega}_{u'})\cap B(0,R\alpha+1) = \supp(\widehat{\omega}_{u})\cap B(0,R\alpha+1),
\]
and therefore 
\[
\begin{aligned}
&\revision{\frac{1}{R}\left|\Exp\left[\left[\dist(0,Re_1;\widehat{\omega}_{u'})- \dist(0,Re_1;\widehat{\omega}_{u})\right]\cdot  \ind_{A^c_{R}}(\widehat{\omega}_{u'})\ind_{A^c_{R}}(\widehat{\omega}_{u})\right]\right|} \\
&\qquad \leq 2 \cdot \prob\brk{\supp(\widehat{\omega}_{u'})\cap B(0,R\alpha+1) \neq \supp(\widehat{\omega}_{u})\cap B(0,R\alpha+1)
}.
\end{aligned}
\]
Assume without loss of generality that $u'\leq u$. Due to the coupling, the point measure $\widehat{\omega}_{u}$ is obtained from $\widehat{\omega}_{u'}$ by adding to it an additional independent point measure in $\Omega$ which is distributed as a Poisson point process with intensity measure $(u-u')\cdot \Leb_d(dx)$. In particular, the probability that there is an additional point inside the ball $B(0,R\alpha+1)$ is 
\[
	1-e^{-(u-u')\kappa_d (R\alpha+1)^d} \leq 1-e^{-\delta \kappa_d (R\alpha+1)^d}.
\]
Thus, with probability at most $1-e^{-\delta \kappa_d (R\alpha+1)^d}$ the point measures $\widehat{\omega}_{u'}$ and $\widehat{\omega}_{u}$ do not coincide inside the ball $B(0,R\alpha+1)$. Recalling that $R$ depends only on $\varepsilon$ and $I$ we can choose  $\delta$ small enough so that $2(1-e^{-\delta \kappa_d (R\alpha+1)^d})<\varepsilon/3$,  thus completing the proof. 

\end{proof1}

\subsection{Volume convergence}

\begin{proposition}\label{prop:vol_esimation}
For every $u\geq 0$, $\prob_u$-almost surely
\begin{equation}\label{eq:vol_estimation}
 \lim_{R\to\infty} \nu_R([0,1]^d)= \lim_{M\to\infty} \frac{\Leb_d([0,M]^d \setminus \cal S(\omega))}{\Leb_d([0,M]^d)} = e^{-u\kappa_d} =\mu_{\sigma(u)}([0,1]^d).
\end{equation}
\end{proposition}

\begin{proof}
As observed in Proposition \ref{prop:upper_bound_on_eta_u} 
\[
	\prob_u(0\notin \calS(\omega)) = e^{-u\kappa_d}. 
\]
Using the ergodicity of the model, see Lemma \ref{lem:ergodicity}, and the ergodic theorem we can conclude that 
\[
\begin{aligned}
	\lim_{M\to\infty} \frac{\Leb_d([0,M]^d \setminus \cal S(\omega))}{\Leb_d([0,M]^d)} & = \lim_{M\to\infty}\frac{1}{\Leb_d(B(0,M))}\int_{B(0,M)} \ind_{x\notin \calS(\omega)}\Leb_d(dx)\\
	& = \Exp_u[\ind_{0\notin \calS(\omega)}] = e^{-u\kappa_d}.
\end{aligned}
\]
\end{proof}

\section{Convergence for uniform point distributions}
\label{sec:convergene_uniform_distribution}

In this section, we exploit the results proved in Sections \ref{sec:Uniform_dist_of_points}--\ref{sec:concentration_results} to prove Parts $1$ and $4$ of \thmref{thm:Main_theorem} for the case when $u$ is uniform and $D$ is convex. These assumptions are relaxed in the next section.
Since for uniform $u$, we have a natural coupling of the measures $\prob_{u,R}$ using a single measure $\prob_u$, we will have in fact a slightly stronger result than stated in \thmref{thm:Main_theorem}.

The main result of this section is the following:

\begin{theorem}\label{thm:conv_uniform_compact_point_defects}
Let $u\in [0,u_*)$ and let $D\subset \R^d$ be a convex, compact $d$-dimensional manifold with corners. Then
\begin{equation}
\lim_{R\to\infty} \sup_{x,y\in D} \left| \boldd_{\eta(u)}^D(x,y) - \dist_R^D(\pi_R(x),\pi_R(y)) \right| =0, \qquad \prob_{u}\text{-a.s.}
\end{equation}
In particular, since $\pi_R$ is onto, the sequence $(D_R,\dist_R^D)$ converges $\prob_u$-a.s.~to $(D,\boldd_{\eta(u)}^D)$ with respect to the Gromov-Hausdorff metric, where $\boldd_{\eta(u)}^D$ should be interpreted as in \eqref{eq:defn_of_d_rho^D}.
\end{theorem}

\begin{proof}
Denote $\calA_R^D = (D_R, \dist_R^D)$, $\calB_R^D = (D_R,\dist_R)$ and $\calC^D =( D,\boldd_{\eta(u)}^D)$.
\revision{We denote by ${\pi}_R^D$ the projection $D\to D_R$ when considered as a mapping between $\calC^D$ and $\calA_R^D$.} 
We will use the space $\calB_R^D$ as an intermediate metric space in order to bound the distortion of $\pi_R^D$. 
We denote by $\tilde{\pi}_R^D$ the projection $D\to D_R$ when considered as a mapping between $\calC^D$ and $\calB_R^D$ and by $\text{Id}$ the identity from $D_R$ to itself when considered as a map between $\calB_R^D$ and $\calA_R^D$. See Figure \ref{fig:three_spaces_illustration} for an illustration. 

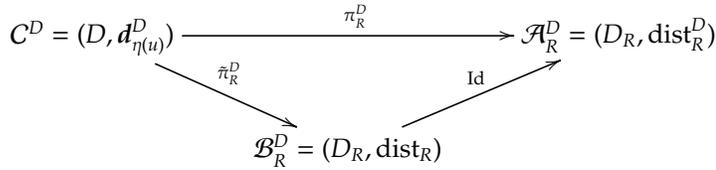
\begin{figure}[h]
\centerline{
\xymatrix{ \calC^D = (D,\boldd_{\eta(u)}^D ) \ar[rr]^{\pi_R^D} \ar[dr]^{\tilde{\pi}_R^D} & & \calA_R^D = (D_R,\dist_R^D) \\
& \calB_R^D = (D_R,\dist_R) \ar[ru]^{\text{Id}}&}
}
\caption{The spaces $\calA_R^D$,$\calB_R^D$ and $\calC^D$.}
\label{fig:three_spaces_illustration}
\end{figure}

By the triangle inequality,
\[
\dis \pi_R^D \le \dis \tilde{\pi}_R^D + \dis \text{Id},
\] 
\revision{where the distortion is defined as in Subsection~\ref{subsec:2.3.6}. Thus,
}
it is enough to prove that with $\prob_u$-probability one, both $\dis \tilde{\pi}_R^D$ and $\dis \text{Id}$ go to zero as $R\to\infty$. 

For $ \dis \tilde{\pi}_R^D$, observe that $D$ is compact and convex, and since $u$ is constant, it follows that $D$ is convex with respect to $\boldd_{\eta(u)}$, hence $\boldd_{\eta(u)}^D = \boldd_{\eta(u)}$. The mapping $\tilde{\pi}_R^D:\calC_D^R \to \calB_D^R$ is onto and by \thmref{thm:distances_in_M_R} has $\prob_u$-a.s.~\revision{an asymptotically} vanishing distortion.

The rest of the proof shows that the distortion of $\text{Id} : \calA_R^D \to \calB_R^D$ also vanishes asymptotically. 
We use the concentration result \propref{pn:concentration_est} as follows: Fix $\delta>0$, and denote by $D(\delta)$ the set of points in $D$ whose distance from the boundary $\partial D$ is greater than $\delta$. For $R>0$, let $N_{R,\delta}$ denote a finite $1/\sqrt{R}$-net of $D(\delta)$ such that $|N_{R,\delta}| < CR^{d/2}$ and for every $x,y\in N_{R,\delta}$, $|x-y|>c/\sqrt{R}$, with $C$ and $c$ depending only on $d$ \revision{and $D$}.
Denote $0<\e = \delta /\revision{(2\diam(D))}$, and let $x,y\in N_{R,\delta}$. By \propref{pn:concentration_est} we have that for $R>R_0(\delta)$,
\begin{equation}
\begin{aligned}
\prob_u\brk{\forall \gamma \in \Gamma_\varepsilon(Rx,Ry) ~:~ d_H(\gamma,[Rx,Ry])>\varepsilon R|x-y|} & \leq C_7 e^{-c_8 (R|x-y|)^{1/2}} \\
& < C_{9} e^{-c_{10} R^{1/4}},
\end{aligned}
\end{equation}
and therefore by a union bound argument \revision{and the fact that $|N_{R,\delta}|^2 \le C\,R^d$}, we have
\beq
\label{eq:estimation_geod_from_seg}
\prob_u\brk{\exists x,y\in N_{R,\delta} \,\forall \gamma \in \Gamma_\varepsilon(Rx,Ry) ~:~ d_H(\gamma,[Rx,Ry])>\varepsilon R|x-y|} \leq C_{11} R^d e^{-c_{10} R^{1/4}}.
\eeq
Considering the sequence of events in \eqref{eq:estimation_geod_from_seg} with $R$ replaced by $m\in\bbN$, we get that the sum of the probabilities is finite, hence by the Borel\revision{--}Cantelli lemma, we have that $\prob_u$-a.s.~there exists $M_0(\delta,\omega)$ such that for every $m\ge M_0(\delta, \omega)$,
\[
\forall x,y\in N_{m,\delta} \,\exists \gamma \in \Gamma_\varepsilon(m x,m y) ~:~ d_H(\gamma,[m x,m y]) \le \e m |x-y| < \frac{m\delta}{2}.
\]

Since $D$ is convex, $[x,y]\subset D$. Therefore by the definition of $\delta$, \revision{it} follows that for every $m\geq M_0$ and every $x,y\in N_{m,\delta}$ there exists an $\e$-geodesic with respect to $\dist_R$ that remains in $D_R$ for $R$ large enough, hence 
\[
\lim_{m\to \infty} \sup_{x,y\in N_{m,\delta}} |\dist_{m}(x,y) - \dist_{m}^D (x,y)| \le \e, \quad \prob_u\text{-a.s.}
\]
Since $N_{m,\delta}$ is \revision{a} $1/\sqrt{m}$-net of $D(\delta)$ \revision{it} follows that
\[
\lim_{m\to \infty} \sup_{x,y\in D(\delta)} |\dist_{m}(x,y) - \dist^D_{m} (x,y)| \le \e, \quad \prob_u\text{-a.s.}
\]
and therefore
\[
\lim_{m\to \infty} \sup_{x,y\in D} |\dist_{m}(x,y) - \dist_{m}^D (x,y)| \le \e + 4\delta = \delta\brk{4+ \frac{1}{2\diam(D)}}, \quad \prob_u\text{-a.s.}
\]
Since for every $R>0$ there exists an $m$ such that $|m-R|<1$, it follows that for such a choice of $m$, for every $x,y\in D$, 
\[
\begin{aligned}
	& |\dist_{m}(x,y) - \dist_R(x,y)| =\left|\frac{\dist(m x,m y)}{m} - \frac{\dist(Rx,Ry)}{R}\right|\\
	& \quad \le \frac{\dist(m x,Rx)}{m} + \frac{\dist(m y,Ry)}{m}+\dist(Rx,Ry)\left|\frac{1}{m}-\frac{1}{R}\right|\\
	& \quad \leq \frac{|x|_\infty}{R-1}+\frac{ |y|_\infty}{R-1} + 
	\frac{|x-y|_\infty}{R-1}\leq \frac{4 \sup_{z\in D}|z|_\infty}{R-1}.
\end{aligned}
\]
Similarly, 
\[
|\dist^D_{m}(x,y) - \dist^D_R(x,y)|\leq \frac{4 \sup_{z\in D}|z|_\infty}{R-1}.
\]
It follows that
\[
	\limsup_{R\to \infty} \sup_{x,y\in D} |\dist_R(x,y) - \dist_R^D (x,y)| \le \e + 4\delta = \delta\brk{4+ \frac{1}{2\diam(D)}}, \quad \prob_u\text{-a.s.}
\]
Since $\delta$ is arbitrary we finally have that
\beq
\label{eq:dist_R_vs_dist_D_R}
\lim_{R\to \infty} \sup_{x,y\in D} |\dist_R(x,y) - \dist_R^D (x,y)| = 0, \quad \prob_u\text{-a.s.}
\eeq
This shows that the identity mapping $\text{Id}: \calB_R^D\to \calA_R^D$ has \revision{an asymptotically} vanishing distortion $\prob_u$-a.s., which completes the proof.
\end{proof}

The following proposition proves Part $4$ of \thmref{thm:Main_theorem} for the case of constant $u$, that is, that $(\pi_R^D)^{-1}:D'_R \to D$ is asymptotically surjective.
\begin{proposition}
\label{pn:asymptotic_surjectivity_constant_u}
	\begin{equation}
		\lim_{R\to\infty} d_H(D,\pi^{-1}_R(D'_R)) =0, \quad \prob_{u}\text{-a.s.} \nonumber
	\end{equation}
\end{proposition}

\begin{proof}
It follows from Lemma \ref{lem:diam_of_clusters_2} that 
\[
	\begin{aligned}
		& \prob_u\brk{\begin{array}{l}
		\exists \text{ a connected component }\calC \text{ in } \bigcup_{x\in[-M,M]^d}
		\calS(x;\omega) \\
		\text{such that } \diam_\euc(\calC)\geq \log^2{M}
		\end{array}
		}\\
		\leq & \sum_{z\in \bbZ^d \cap [-M\revision{-1},M]^d} \prob_u\brk{\begin{array}{l}\exists \text{ a connected 	component }\calC \text{ in } \bigcup_{x\in z+[0,1]^d}\calS(x;\omega) \\
		\text{such that } \diam_\euc(\calC)\geq \log^2{M}
		\end{array}
		}\\
		\leq & C_5(2M+\revision{2})^de^{-c_6 \log^2(M)}.
	\end{aligned}
\]
\revision{The choice of $\log^2M$ is dictated by the need of satisfying two conditions:  we need a term which is $o(M)$, and the probability must decay sufficiently fast; see \eqref{eq:size_of_clusters_in_a_box}.}

Taking $M=KR$, with $K=\diam_\euc(D)$ and $R>0$, we get that with $\prob_u$-probability at least $1-C_5(2KR+1)^d e^{-c_6\log^2(KR)}$ the distance \revision{from} any point in $\calS \cap R\cdot D$ to $R\cdot D\setminus \calS$ is at most $\log^2(KR)$. Thus
\begin{equation}\label{eq:size_of_clusters_in_a_box}
	\prob_u\brk{
	d_{H}(D,(\pi_R^D)^{-1}(D'_R)) \ge \frac{\log^2 (KR)}{R}	
	}\leq C_5(2KR+\revision{2})^de^{-c_6 \log^2(KR)}.
\end{equation}

Given $\varepsilon>0$, consider the sequence $\brk{d_{H}(D,(\pi_{m\varepsilon}^D)^{-1}(D'_{m\varepsilon}))}_{m\geq 1}$. Since the righthand side of \eqref{eq:size_of_clusters_in_a_box} is summable in $m$, it follows from the Borel--Cantelli lemma, that $\prob_u$-almost surely $d_{H}(D,(\pi_{m\varepsilon}^D)^{-1}(D'_{m\varepsilon}))\leq \frac{\log^2(K\varepsilon m)}{\varepsilon m}$  for all but finitely many $m$'s and in particular that 
\[
	\lim_{m\to\infty} d_{H}(D,(\pi_{m\varepsilon}^D)^{-1}(D'_{m\varepsilon})) =0.
\]
Since for every $0<R\leq S$ we have $d_{H}((\pi_{R}^D)^{-1}(D'_{R}),(\pi_{S}^D)^{-1}(D'_{S}))\leq |R-S|$ and since the sequence $\brk{d_{H}(D,(\pi_{m\varepsilon}^D)^{-1}(D'_{m\varepsilon}))}_{m\geq 1}$ is $\varepsilon$\revision{-}dense in $[0,\infty)$ it follows that $\prob_u$-a.s.~$\limsup_{R\to\infty} d_{H}(D,(\pi_{R}^D)^{-1}(D'_{R})) \leq \varepsilon$. Since $\varepsilon>0$ is arbitrary we get that $\prob_u$-a.s.~$\lim_{R\to\infty} d_{H}(D,(\pi_{R}^D)^{-1}(D'_{R}))=0$.
\end{proof}



\section{Proof of \thmref{thm:Main_theorem}}
\label{sec:proof_main_thm}

In this section we prove Parts 1,2 and 4 of \thmref{thm:Main_theorem}. Part 3 was proved in \propref{prop:upper_bound_on_eta_u}.
Let $D\subset \R^d$ a compact $d$-dimensional manifold with corners, and let $u:D \to [0,u_*)$ be a continuous function. Since $D$ is compact, we can always extend $u$ continuously to $\R^d$ without enlarging its upper bound. Therefore, we can assume without loss of generality that $u:\R^d\to [0,u_*)$, with $\sup u<u_*$. The parameters
$u$ and $R>0$ define a process with probability measure $\prob_{u,R}$.
Similarly to the proof of \propref{prop:continuity_of_eta_u}, we start by introducing a natural coupling of the probability measures $\prob_{u,R}$ for a given $R>0$.

Let 
\[
	\hOmega = \BRK{\homega=\sum_{i\geq 0} \delta_{(x_i,u_i)} ~:~ \begin{array}{l}
	x_i \in\R^d,~u_i\in [0,\infty) \text{ for all } i\geq 0 \text{ and } \homega(A\times [0,u])<\infty \\
	\text{for all compact } A\subset \R^d \text{ and } u\geq 0 
	\end{array}}
\]
and for a continuous function $u:\R^d\to [0,\infty)$ and $\homega\in\hOmega$ define 
\[
	\homega_u = \sum_{{\scriptsize \begin{array}{c}(x_i,u_i)\in \text{supp}\homega,\\u_i\leq u(x_i)\end{array}}}\delta_{x_i}.
\]

Finally, let $\prob_R$ be the probability measure on $\hOmega$ under which $\homega$ is distributed like a Poisson point process on $\R^d\times [0,\infty)$ with intensity $R^d\Leb_d(dx)\times \Leb_1(dx)|_{[0,\infty)}$. 

One can now verify, that for every continuous $u:\R^d\to [0,\infty)$, the distribution of $\omega$ under $\prob_{u,R}$ is the same as the distribution of $\homega_u$ under $\prob_R$. That is,  we constructed a coupling of the probability measures $\prob_{u,R}$,  for $u:\R^d\to [0,\infty)$, under which
\begin{equation}\label{eq:coupling_R}
	\homega_u\leq \homega_{u'},\quad \forall  u,u':\R^d\to [0,\infty),~ u\le u'.\nonumber
\end{equation}

\subsection{Metric convergence}

In this section we prove Part 1 of \thmref{thm:Main_theorem}, which with the coupling constructed above states that:
\begin{quote}
{\em For every $\e>0$,
\beq
\label{eq:dis_pi_R}
\lim_{R\to\infty} \prob_{R} \brk{ \sup_{x,y\in D} \left| \boldd_{\eta\circ u}^D(x,y) - \dist_R^D(\pi_R(x),\pi_R(y);\widehat{\omega}_u) \right| < \varepsilon }=1.
\eeq
}
\end{quote}

In particular, this implies that for every $\e>0$,
\[
\lim_{R\to\infty} \prob_{R} (d_{GH} ((D_R,\dist_R^D(\cdot;\widehat{\omega}_u)),(D,\boldd_{\eta\circ u}^D)\revision{)} > \e) = 0.
\]
\thmref{thm:conv_uniform_compact_point_defects} states that this holds when $D$ is convex and $u=u_0$ is uniform.


We start with a proposition showing that if $D$ is convex and $u:D\to [0,u_*)$ only takes values within a small interval $[u_{\min},u_{\max}]$, then we can bound (with high probability) the distortion of the projection $\pi_R$ between $(D,\boldd_{\eta\circ u}^D)$ and $(D_R(\homega_u),\dist_R^D(\cdot;\widehat{\omega}_u))$.

\begin{proposition}
\label{prop:u_continuity}
Assume that $D$ is a compact, path-connected and convex set \revision{with a non-empty interior}, and let $u_{\min} = \min_D u \ge 0$ and $u_{\max} = \max_D u < u_*$. Consider $\pi_R$ as a function $(D,\boldd_{\eta\circ u}^D) \to (D_R(\homega_u),\dist_R^D(\cdot;\widehat{\omega}_u))$.
Then
\[
\lim_{R\to\infty} \prob_{R} (\dis \pi_R > 3\Delta \eta\cdot  \diam_\euc D) = 0,
\]
where $\Delta \eta = \eta(u_{\min}) - \eta(u_{\max})$.
\end{proposition}

\begin{proof}
Consider the coupling as above between $\homega_u$, $\homega_{u_{\min}}$ and $\homega_{u_{\max}}$, with $u_{\min}$ and $u_{\max}$ viewed as constant functions. Then,
\[
\dist_R^D (\cdot;\homega_{u_{\min}}) \ge \dist_R^D (\cdot;\homega_{u}) \ge \dist_R^D (\cdot;\homega_{u_{\max}})
\quad \text{and} \quad
\boldd^D_{\eta(u_{\min})} \ge \boldd^D_{\eta\circ u} \ge \boldd^D_{\eta(u_{\max})},
\]
and in particular,
\beq
\label{eq:distmin_distmax1}
|\dist_R^D (\cdot;\homega_{u_{\min}}) - \dist_R^D (\cdot;\homega_{u})| \le |\dist_R^D (\cdot;\homega_{u_{\min}}) - \dist_R^D (\cdot;\homega_{u_{\max}})|,
\eeq
and
\beq
\label{eq:distmin_distmax2}
|\boldd^D_{\eta(u_{\min})} - \boldd^D_{\eta\circ u}| \le |\boldd^D_{\eta(u_{\min})} - \boldd^D_{\eta(u_{\max})}| \le \Delta \eta \cdot \diam_\euc D.
\eeq
Using \eqref{eq:distmin_distmax1} and \eqref{eq:distmin_distmax2}, and by a repeated application of the triangle inequality,
\beq
\label{eq:dis_pi_R_convex}
\begin{split}
\dis \pi_R &= \sup_{x,y\in D} |\boldd^D_{\eta\circ u}(x,y) - \dist_R^D(x,y;\homega_u)| \\
& \le \sup_{x,y\in D} |\boldd^D_{\eta\circ u}(x,y) - \boldd^D_{\eta(u_{\min})}(x,y)| + |\boldd^D_{\eta(u_{\min})}(x,y) - \dist^D_R(x,y;\homega_{u_{\min}})| \\
&\qquad + |\dist^D_R(x,y;\homega_{u_{\min}}) - \dist_R^D(x,y;\homega_u)| \\
& \le \sup_{x,y\in D}  |\boldd^D_{\eta(u_{\max})}(x,y) - \boldd^D_{\eta(u_{\min})}(x,y)| + |\boldd^D_{\eta(u_{\min})}(x,y) - \dist^D_R(x,y;\homega_{u_{\min}})| \\
&\qquad+  |\dist^D_R(x,y;\homega_{u_{\min}}) - \dist^D_R(x,y;\homega_{u_{\max}})| \\
& \le \sup_{x,y\in D}  2|\boldd^D_{\eta(u_{\max})}(x,y) - \boldd^D_{\eta(u_{\min})}(x,y)| + 2 |\boldd^D_{\eta(u_{\min})}(x,y) - \dist^D_R(x,y;\homega_{u_{\min}})| \\
&\qquad +  |\boldd_{\eta(u_{\max})}(x,y) - \dist^D_R(x,y;\homega_{u_{\max}})| \\
& \le 2\Delta\eta\cdot \diam_\euc D +2 |\boldd^D_{\eta(u_{\min})}(x,y) - \dist^D_R(x,y;\homega_{u_{\min}})| \\
&\qquad +  |\boldd_{\eta(u_{\max})}(x,y) - \dist^D_R(x,y;\homega_{u_{\max}})|.
\end{split}
\eeq
The result now follows by applying \thmref{thm:conv_uniform_compact_point_defects} to the last two addends. $~~$
\end{proof}


We now prove Part 1 of \thmref{thm:Main_theorem}. 
The idea behind the proof is the following: We partition $D$ into small convex sets, in each $u$ varies only a little.
\propref{prop:u_continuity} states that with high probability, the distortion of the projection in each set is small.
We then glue the sets together and show that the accumulated distortion remains small.
A technical complication arises when $D$ cannot be partitioned into finitely many convex sets. We overcome this problem by considering sets slightly larger and slightly smaller than $D$, denoted by $\overline{D}$ and $\underline{D}$, that can be partitioned in such a way.

\paragraph{Step I: Partitioning $D$.}

Let $n$ be a large natural number to be chosen later (independent of $R$). Cover $D$ with cubes of edge length $1/n$, whose corners are on the lattice $\frac{1}{n}\bbZ^d$; henceforth, ``vertices'' refers to the corners of the cubes.  
Denote the cubes that intersect $D$ but not $\partial D$ by $\Box_{n,1},\ldots, \Box_{n,k_n}$, and those that intersect $\partial D$ by $\Box_{n,k_n+1},\ldots, \Box_{n,k_n+m_n}$. Since $D$ is compact, there exists a slightly larger compact set $\Box$ containing $\bigcup_{i=1}^{k_n + m_n} \Box_{n,i}$ for all $n\geq 1$. The function $\eta\circ u$ varies on each cube $\Box_{n,i}$, by some $\Delta_{n,i}$. Denote $\Delta_n = \max_i \Delta_{n,i}$. Since $\eta\circ u$ is continuous (\propref{prop:continuity_of_eta_u}) and $\Box$ is compact, $\Delta_n\to 0$ as $n\to\infty$.
Denote by $T_n$ the union of the facets of the cubes and let
\[
\underline{D} = \bigcup_{i=1}^{k_n} \Box_{n,i},\qquad \overline{D} = \bigcup_{i=1}^{k_n + m_n} \Box_{n,i}.
\]
It follows from the definition of $\underline{D}$ and $\overline{D}$ that $\underline{D}\subset D \subset \overline{D}$ and both $|\diam \overline{D} - \diam D|$ and $|\diam \underline{D} - \diam D|$ are of order $1/n$. Moreover, we claim that
\beq
\label{eq:dist_barD_underlineD}
\lim_{n\to \infty} \sup_{x,y\in \underline{D}} |\boldd_{\eta\circ u}^{\overline{D}}(x,y) -  \boldd_{\eta\circ u}^{\underline{D}}(x,y) | = 0.
\eeq
Indeed, note that $\boldd_{\eta\circ u}^{\overline{D}} \le \boldd_{\eta\circ u}^{\underline{D}}$ for \revision{any} pair of points in $\underline{D}$, and therefore we only need to prove that for every $x,y\in\underline{D}$ and every simple curve $\gamma\subset \overline{D}$ between $x$ and $y$, there exists a simple curve $\gamma' \subset \underline{D}$ between $x$ and $y$ such that 
\begin{equation}\label{eq:dist_barD_underline_paths}
	\len_{\boldd_{\eta\circ u}}(\gamma') < \len_{\boldd_{\eta\circ u}} (\gamma) + o(1),
\end{equation} 
where $o(1)$ is with respect to $n$, independent of $R$, \revision{$x$ and $y$}. To prove \eqref{eq:dist_barD_underline_paths}, note that since $D\subset \R^d$ is compact,  $\overline{D}$ and $\underline{D}$ are homotopic whenever $n$ is large enough ($\partial D$ is a compact submanifold with corners and therefore its $1/n$-neighborhood is homotopic to itself for $n$ large enough) and $d_H(\overline{D},\underline{D}) < \sqrt{d}/n$. 
Then, for every simple curve $\gamma\subset \overline{D}$ between $x$ and $y$ in $\underline{D}$, there exists a simple curve $\gamma' \subset \underline{D}$ between $x$ and $y$ such that the $d_H (\gamma,\gamma') < 2\sqrt{d}/n$ and $\len_\euc (\gamma') < \len_\euc(\gamma) + c/n$ for some $c=c(d,D)>0$ (see \figref{fig:section8}).
Since $\eta\circ u$ is continuous, $\gamma$ and $\gamma'$ are simple, and their Hausdorff distance is $o(1)$, \eqref{eq:dist_barD_underline_paths} follows.

\begin{figure}
\begin{center}
\includegraphics[height=2.5in]{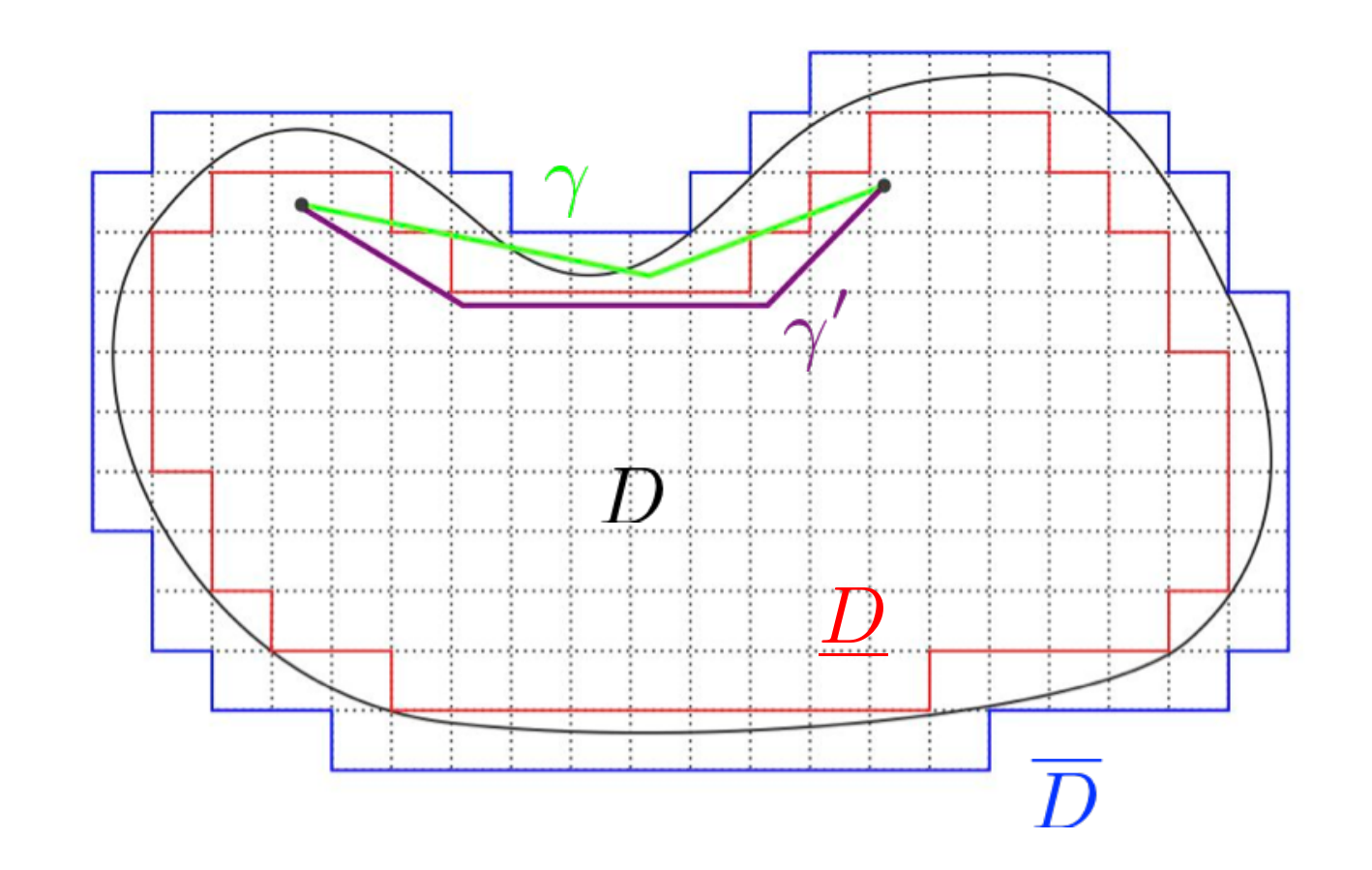}
\end{center}
\caption{Illustration of the construction of step I of the proof. The boundary of the domain $D$ is marked in black, and the dotted lines represent the grid that forms $\Box_{n,i}$. The boundaries of $\overline{D}$ and $\underline{D}$ are marked with blue and red lines, respectively. For every simple curve $\gamma\subset\overline{D}$ between two points in $\underline{D}$ there is a curve $\gamma'\subset \underline{D}$ of similar length which is close to $\gamma$ in the Hausdorff distance (see \eqref{eq:dist_barD_underline_paths}). Therefore, the inclusion of $\underline{D}$ in $\overline{D}$ has a vanishing distortion as $n\to\infty$, as stated in \eqref{eq:dist_barD_underlineD}.}
\label{fig:section8}
\end{figure}

\paragraph{Step II: Each geodesic \revision{(w.r.t.~either $\boldd_{\eta\circ u}^D$ or $\dist_R^D$)} intersects only $O(n)$ cubes (in each of which the distortion is small).}

Consider a vertex of one of the cubes, and a facet of the same cube that does not intersect it. Since distances with respect to $\boldd_{\eta\circ u}$ are bounded from below by $\underline{\eta} = \min_{\overline{D}} \eta(u)$ times the Euclidean distances, the distance between the vertex and the facet is at least $\underline{\eta}/n$.
Therefore, a ball in $(D,\boldd_{\eta\circ u}^D)$ of radius $\underline{\eta}/2n$  intersects at most $2^d$ cubes. 
It follows that all geodesics in $(D,\boldd_{\eta\circ u}^D)$  (or more generally, every curve of length $\diam_\euc D$ or less) intersect at most 
\[
\frac{\diam_\euc D}{\underline{\eta}/2n} \cdot 2^d = \frac{2^{d+1} \diam_\euc D}{\underline{\eta}} \cdot n
\] 
of the cubes $\Box_{n,1},\ldots, \Box_{n,k_n+m_n}$ (see Lemma 5.3 in \cite{KM15b} for a similar argument). By the same reasoning, each geodesic in either $(\overline{D},\boldd_{\eta\circ u}^{\overline{D}})$ or $(\underline{D},\boldd_{\eta\circ u}^{\underline{D}})$ intersects at most $2^{d+1} n \diam\revision{_\euc} D/\underline{\eta}$ cubes.

Next, cover similarly $D_R$  by ``cubes" $\Box_{n,i}^R = \pi_R(\Box_{n,i})$; also, define $\underline{D}_R$ and $\overline{D}_R$ in an analogous way.
Applying \propref{prop:u_continuity} to each cube $\Box_{n,i}$, we obtain that 
\beq
\label{eq:cubes_GH_conv}
\lim_{R\to\infty} \prob_{u,R} (\dis \pi_R|_{\Box_{n,i}} < 3\sqrt{d} \Delta_n/n \revision{\quad \forall i=1,\dots,k_n+m_n}) = 1.
\eeq

For a given $R$, denote 
\[
A_R=\left\{\dis \pi_R|_{\Box_{n,i}} < 3\sqrt{d} \Delta_n/n \revision{\quad \forall i=1,\dots,k_n+m_n}\right\}. 
\]
Assume that $A_R$ holds, and consider a vertex in one of the cubes in $\overline{D}_R$, and a facet of the same cube that does not intersect the vertex. 
The distance between the vertex and the facet is at least $(\underline{\eta} - 3\sqrt{d} \Delta_n )/n$, which is positive for large enough $n$. 
By the same reasoning as above, each geodesic in either $\underline{D}_R$ or $\overline{D}_R$ intersects at most $2^{d+1} n \diam D/(\underline{\eta} - 3\sqrt{d} \Delta_n)$ cubes. 

\paragraph{Step III: Bounding the distortion.}

We now want to show that when $A_R$ holds, i.e., the distortion within each cube is small, then the distortion of $\pi_R:(\overline{D},\boldd_{\eta\circ u}^{\overline{D}})\to (\revision{\overline{D}_R},\dist_R^{\overline{D}})$ is small as well (and similarly with $\underline{D}$). The idea,  which is similar to the proof of Theorem 3.1 in \cite{KM15}, is that a geodesic $\gamma$ in $(\overline{D},\boldd_{\eta\circ u}^{\overline{D}})$ crosses $O(n)$ cubes at points $x_1,\ldots x_m$, and at each crossing it accumulates a distance $\boldd_{\eta\circ u} (x_i, x_{i+1})$. When $A_R$ holds,
\[
|\boldd_{\eta\circ u} (x_i, x_{i+1}) - \dist_R^D (\pi_R (x_i), \pi_R(x_{i+1}))| \le \frac{3\sqrt{d}\Delta_n}{n},
\]
i.e.~the distortion accumulated in each cube is $O(\Delta_n/n)$, so the total distortion is $O(n)\cdot O(\Delta_n/n) = O(\Delta\revision{_n}) \to 0$. \revision{Here and below, the constants in $O(\cdot)$ only depend on $d$, $\diam_\e D$ and $\underline{\eta}$, which are fixed throughout the proof.}

Formally, let $x,y\in \overline{D}$, and let $\gamma$ be a geodesic in $(\overline{D},\boldd_{\eta\circ u}^{\overline{D}})$ between $x$ and $y$. If $\gamma$ does not intersect the facets of the boxes transversely, we can take $\gamma$ to be infinitesimally longer such that it does. Therefore, without loss of generality $\gamma$ intersects the facets of the boxes at a finite number of points $x_1,\ldots,x_m$, with $m < 2^{d+1} n \diam D / \underline{\eta} = O(n)$. Denote $x=x_0$ and $y=x_{m+1}$.
When the event $A_R$ holds,
\[
\begin{split}
\dist_R^{\overline{D}} (\pi_R (x), \pi_R(y) ) & \le \sum_{i=1}^{m} \dist_R^D (\pi_R (x_i), \pi_R(x_{i+1}) ) \le \sum_{i=1}^{m} \brk{ \boldd_{\eta\circ u} (x_i, x_{i+1}) + \frac{3\sqrt{d}\Delta_n}{n} } \\
	& \le  \sum_{i=1}^{m} \boldd_{\eta\circ u}^{\overline{D}} (x_i, x_{i+1}) + m \cdot O(\Delta_n/n) = \len_{\boldd_{\eta\circ u}^{\overline{D}}} (\gamma) + O(\Delta_n)\\
	& = \boldd_{\eta\circ u}^{\overline{D}}(x,y) + O(\Delta_n).
\end{split}
\]
A similar argument, on a geodesic $\sigma$ between $\pi_R (x)$ and $\pi_R(y)$, shows that
\[
\boldd_{\eta\circ u}^{\overline{D}}(x,y) \le \dist_R^{\overline{D}} (\pi_R (x), \pi_R(y) ) + O(\Delta_n).
\]
Therefore, we obtain that when $A_R$ holds,
\beq
\label{eq:dis_barD}
\sup_{x,y\in \overline{D}} |\boldd_{\eta\circ u}^{\overline{D}}(x,y) -  \dist_R^{\overline{D}} (\pi_R (x), \pi_R(y) )| \le C \Delta_n,
\eeq
for some $C>0$ independent of $n$.
Similarly, when $A_R$ holds,
\beq
\label{eq:dis_underlineD}
\sup_{x,y\in \underline{D}} |\boldd_{\eta\circ u}^{\underline{D}}(x,y) -  \dist_R^{\underline{D}} (\pi_R (x), \pi_R(y) )| \le C \Delta_n.
\eeq

If $D=\overline{D}$, then the proof is complete. The rest of the proof deals with the case where $\underline{D}\subsetneq D\subsetneq \overline{D}$.
In $\underline{D}$ we have that 
$\boldd_{\eta\circ u}^{\overline{D}} \le \boldd_{\eta\circ u}^{D} \le \boldd_{\eta\circ u}^{\underline{D}}$,
hence
\beq
\label{eq:dist_overline_underline1}
|\boldd_{\eta\circ u}^{\underline{D}} - \boldd_{\eta\circ u}^{D}| \le |\boldd_{\eta\circ u}^{\underline{D}} - \boldd_{\eta\circ u}^{\overline{D}}|.
\eeq
Similarly, $\dist_R^{\overline{D}} \le \dist_R^{D} \le \dist_R^{\underline{D}}$  in $\underline{D}_R$, hence
\beq
\label{eq:dist_overline_underline2}
|\dist_R^{\underline{D}} - \dist_R^{D}| \le |\dist_R^{\underline{D}} - \dist_R^{\overline{D}}|.
\eeq
Using \eqref{eq:dist_barD_underlineD}, \eqref{eq:dis_barD}--\eqref{eq:dist_overline_underline2}, we obtain (similarly as in  \eqref{eq:dis_pi_R_convex}) that when $A_R$ holds,
\[
\label{eq:dis_D}
\begin{split}
&\sup_{x,y\in \underline{D}} |\boldd_{\eta\circ u}^{D}(x,y) -  \dist_R^{D} (\pi_R (x), \pi_R(y) )|  \\ 
	&\qquad \le \sup_{x,y\in \underline{D}} |\boldd_{\eta\circ u}^{D}(x,y) -  \boldd_{\eta\circ u}^{\underline{D}}(x,y) | 
								+ | \boldd_{\eta\circ u}^{\underline{D}}(x,y) -  \dist_R^{\underline{D}} (\pi_R (x), \pi_R(y) )| \\
	&\qquad \qquad				+ |\dist_R^{\underline{D}} (\pi_R (x), \pi_R(y) ) -  \dist_R^{D} (\pi_R (x), \pi_R(y) )| 	\\
	&\qquad \le \sup_{x,y\in \underline{D}} |\boldd_{\eta\circ u}^{\overline{D}}(x,y) -  \boldd_{\eta\circ u}^{\underline{D}}(x,y) | 
								+ | \boldd_{\eta\circ u}^{\underline{D}}(x,y) -  \dist_R^{\underline{D}} (\pi_R (x), \pi_R(y) )| \\
	&\qquad \qquad				+ |\dist_R^{\underline{D}} (\pi_R (x), \pi_R(y) ) -  \dist_R^{\overline{D}} (\pi_R (x), \pi_R(y) )| \\
	&\qquad \le \sup_{x,y\in \underline{D}} 2|\boldd_{\eta\circ u}^{\overline{D}}(x,y) -  \boldd_{\eta\circ u}^{\underline{D}}(x,y) | 
								+ 2| \boldd_{\eta\circ u}^{\underline{D}}(x,y) -  \dist_R^{\underline{D}} (\pi_R (x), \pi_R(y) )| \\
	&\qquad \qquad				+ |\boldd_{\eta\circ u}^{\overline{D}}(x,y) -  \dist_R^{\overline{D}} (\pi_R (x), \pi_R(y) )| 	\\
	&\qquad \le \sup_{x,y\in \underline{D}} 2|\boldd_{\eta\circ u}^{\overline{D}}(x,y) -  \boldd_{\eta\circ u}^{\underline{D}}(x,y) | 
								+ 3C\Delta_n = o(1), 
\end{split}
\]
where $o(1)$ here is with respect to $n$ and independent of $R$. The first inequality is a triangle inequality; the second follows from \eqref{eq:dist_overline_underline1} and \eqref{eq:dist_overline_underline2}; the third is again a triangle inequality; and the last one follows from  \eqref{eq:dis_barD} and \eqref{eq:dis_underlineD}.

Since $\underline{D}$ is a $1/n$-net in $D$ and $\underline{D}_R$ is a $1/n$-net in $D_R$, we obtain that for $\pi_R: (D,\boldd_{\eta\circ u}^{D} ) \to (D_R,\dist_R^{D})$,
\beq
\label{eq:dis_pi_R_final}
\dis \pi_R = o(1).
\eeq

\revision{Finally, let $\e>0$, and choose $n$ large enough such that $\dis \pi_R < \e$ when $A_R$ holds.
Note that this choice of $n$ is independent of $R$.
It follows that as $R\to\infty$,}
\[
\prob_{u,R}( d_{GH}((D_R,\dist_R^D), (D,\boldd_{\eta\circ u}^D)) > \e ) \le \prob_{u,R}( \dis \pi_R > \e ) \le \prob_{u,R}( A_R^c) \to 0.
\]
\hfill $\blacksquare$

\subsection{Measure convergence}

In this section we prove part 2 of \thmref{thm:Main_theorem}:
\begin{quote}
{\em For every $\e>0$,
		\[
			\lim_{R\to\infty} \prob_{u,R} \brk{ \exists f\in W(D),\,\, \text{s.t.}\,\, \left| \int_D f \, d\mu_{\sigma\circ u}  - \int_{D_R} f\circ\pi_R^{-1} \, d\nu_R \right| >\varepsilon }=0,
		\]
		where 
		\[
		W(D) = \BRK{ f\in C(D) : \|f\|_\infty \le 1 \,\,\text{and}\,\, \text{Lip}(f) \le 1 },
		\]
		$\sigma(u) = e^{-u \kappa_d/d}$, and $\kappa_d$ is the volume of the $d$-dimensional unit ball.
}
\end{quote}

We start with a proposition which is a measure analog of \propref{prop:u_continuity}.
\begin{proposition}
\label{prop:u_continuity_measure}
Let $D$ be a cube in $\R^d$ and let $u_{\min} = \min_D u \ge 0$ and $u_{\max} = \max_D u < u_*$. Then
\[
\lim_{R\to\infty} \prob_{u,R} \brk{ \left| \nu_R(D_R) - \mu_{\sigma\circ u}(D) \right| > 5\Delta\sigma \Leb_d(D)} = 0,
\]
where $\Delta \sigma = \sigma(u_{\min})^d - \sigma(u_{\max})^d$.
\end{proposition}

\begin{proof}
As in the proof of \propref{prop:u_continuity} consider the coupling between $\homega_u$, $\homega_{u_{\min}}$ and $\homega_{u_{\max}}$.
Since $u_{\min}\le u\le u_{\max}$, it follows that $\nu_R (\cdot;\homega_{u_{\min}})\ge \nu_R(\cdot; \homega_{u}) \ge \nu_R (\cdot; \homega_{u_{\max}})$.
Then, 
\[
\begin{split}
&\left| \nu_R(D_R;\homega_{u}) - \mu_{\sigma\circ u}(D) \right|\\
	&\quad \le \left| \nu_R(D_R;\homega_{u}) - \nu_R (D_R; \homega_{u_{\max}}) \right| + \left| \nu_R (D_R; \homega_{u_{\max}}) - \mu_{\sigma(u_{\max})}(D) \right| + \left| \mu_{\sigma(u_{\max})}(D) - \mu_{\sigma\circ u}(D) \right| \\
	&\quad \le \left| \nu_R (D_R; \homega_{u_{\min}}) - \nu_R (D_R; \homega_{u_{\max}}) \right| + \left| \nu_R (D_R; \homega_{u_{\max}}) - \mu_{\sigma(u_{\max})}(D) \right| + \left| \mu_{\sigma(u_{\max})}(D) - \mu_{\sigma(u_{\min})}(D) \right| \\
	&\quad \le \left| \nu_R (D_R; \homega_{u_{\min}}) - \mu_{\sigma(u_{\min})}(D) \right| + 2\left| \nu_R (D_R; \homega_{u_{\max}}) - \mu_{\sigma(u_{\max})}(D) \right| + 2\left| \mu_{\sigma(u_{\max})}(D) - \mu_{\sigma(u_{\min})}(D) \right| \\
	&\quad \le \left| \nu_R (D_R; \homega_{u_{\min}}) - \mu_{\sigma(u_{\min})}(D) \right| + 2\left| \nu_R (D_R; \homega_{u_{\max}}) - \mu_{\sigma(u_{\max})}(D) \right| + 2\Delta\sigma \Leb_d(D),
\end{split}
\]
from which the result follows by \propref{prop:vol_esimation}.
\end{proof}

We now prove part \revision{2} of \thmref{thm:Main_theorem}.
In order to simplify the notation, we will consider $\nu_R$ as a measure on $D$ (assigning zero measure to $\calS_R$), rather than on $D_R$. In this way there is no composition with $\pi_R^{-1}$.

As in the proof of the first part, partition $D$ into cubes $\Box_{n,1},\ldots, \Box_{n,k_n+m_n}$, and let $\overline{D} = \bigcup_{i=1}^{k_n + m_n} \Box_{n,i}$.
We extend both $\mu_{\sigma\circ u}$ and $\nu_R$ to $\overline{D}$ by setting $\overline{D}\setminus D$ to be a null-set.

Let $\e>0$ and choose $n>\sqrt{d}/\e$ large enough such that $\Delta \sigma<\e$ in each $\Box_{n,i}\cap D$ and $\Leb_d(\overline{D}) < 2\Leb_d(D)$.
By applying \propref{prop:u_continuity_measure} to each $\Box_{n,i}$ we obtain that
\[
\lim_{R\to\infty} \prob_{u,R} (G_R) = 1,
\]
where 
\[
G_R = \BRK{ \left| \nu_R(\Box_{n,i}) - \mu_{\sigma\circ u}(\Box_{n,i}) \right| < 5\e \Leb_d(\Box_{n,i})\,\, \forall i=1,\ldots\revision{,} k_n+m_n}.
\]

Let $f\in W(D)$. Due to the choice of $n$ and since $\text{Lip}(f) \le 1$, it follows that $\sup_{\Box_{n,i}\cap D} f - \inf_{\Box_{n,i}\cap D} f < \e$ for every $i$. Denote $f_i =  \inf_{\Box_{n,i}\cap D} f$.
Assume that $G_R$ holds. Then,
\beq
\begin{split}
&\left| \int_D f \, d\mu_{\sigma\circ u}  - \int_D f\, d\nu_R \right| 
	\le \sum_{i=1}^{k_n+m_n} \left| \int_{\Box_{n,i}} f \, d\mu_{\sigma\circ u}  -  \int_{\Box_{n,i}} f\, d\nu_R \right| \nonumber\\
	&\qquad \le \sum_{i=1}^{k_n+m_n} \left| \int_{\Box_{n,i}} f \, d\mu_{\sigma\circ u}  -  f_i \,\mu_{\sigma\circ u}(\Box_{n,i}) \right| + |f_i| \left| \,\mu_{\sigma\circ u}(\Box_{n,i})  - \,\nu_R(\Box_{n,i}) \right| + \left| f_i \,\nu_R(\Box_{n,i})  -  \int_{\Box_{n,i}} f\, d\nu_R \right|\nonumber\\
	&\qquad \le \sum_{i=1}^{k_n+m_n} 2\e\Leb_d(\Box_{n,i}) +\|f\|_\infty \left| \,\mu_{\sigma\circ u}(\Box_{n,i})  - \,\nu_R(\Box_{n,i}) \right|  \nonumber\\
	&\qquad \le 7\e \sum_{i=1}^{k_n+m_n} \Leb_d(\Box_{n,i}) = 7\e \Leb_d(\overline{D}) < 14\e \Leb_d(D).  \nonumber
\end{split}
\eeq
Therefore, \revision{as $R\to\infty$,}
\[
\prob_{u,R} \brk{ \exists f\in W(D),\,\, \text{s.t.}\,\, \left| \int_D f \, d\mu_{\sigma\circ u}  - \int_{D} f \, d\nu_R \right| >14\e \Leb_d(D) } \le \prob_{u,R} (G_R^{c}) \to 0,
\]
which completes the proof. \hfill $\blacksquare$

\subsection{Asymptotic surjectivity}

In this section we prove Part 4 of \thmref{thm:Main_theorem}:
\begin{quote}
{\em
	For every $\e>0$,
		\begin{equation}
		\label{eq:D_R_asymptotically_dense}
			\lim_{R\to\infty} \prob_{u,R} \brk{ d_H(D,\pi^{-1}_R(D'_R)) >\varepsilon }=0,
		\end{equation}
	where $D'_R = \{x\in D_R ~:~ |\pi_R^{-1}(x)|=1\}$ and $d_H$ is the Hausdorff distance in $\R^d$.
}
\end{quote}

This is an immediate consequence of \propref{pn:asymptotic_surjectivity_constant_u}, using the above coupling between $u$ and $u_{\max} = \max_D u$. 
Denote 
$\overline{D}'_R = \{x\in D_R(\homega_{u_{\max}}\revision{)} ~:~ |\pi_R^{-1}(x;\homega_{u_{\max}})|=1\}$. 
Then 
	\beq
	\label{eq:D_R_and_barD_R}
		\pi^{-1}_R(\overline{D}'_R) \subset \pi^{-1}_R(D'_R) \subset D.
	\eeq
By \propref{pn:asymptotic_surjectivity_constant_u} for every $\e>0$
	\begin{equation}
	\label{eq:barD_R_asymptoticaly_dense}
			\lim_{R\to\infty} \prob_{R} \brk{ d_H(D,\pi^{-1}_R(\overline{D}'_R)) >\e }=0,
	\end{equation}
and combining \eqref{eq:barD_R_asymptoticaly_dense} with \eqref{eq:D_R_and_barD_R} we obtain \eqref{eq:D_R_asymptotically_dense}. \hfill $\blacksquare$

\subsection{Proof of \corrref{cor:mGH_convergence}}

We follow here definitions and notations of \defref{def:mGH}.

Since $\pi_R:D\to D_R$ is always defined on the whole space $D$ and always onto $D_R$, Part $1$ of \thmref{thm:Main_theorem} as stated in \eqref{eq:dis_pi_R} immediately implies \eqref{eq:GHconvergence}, which proves the first part of \corrref{cor:mGH_convergence}.

For the second \revision{half of the first} part, we need to prove that 
\begin{enumerate}[label=(\roman*)]
	\item $\pi_R^{-1} : D'_{R_n}\to D$ is an $\e_n$-approximation 
	\item For every $f\in C(D)$,
		\beq
		\label{eq:measure_convergence_mu_R}
			\lim_{n\to\infty} \int_{D'_{R_n}} f\circ \pi_{\revision{R_n}}^{-1} \,d\nu_{\revision{R_n}} = \int_D f \, d\mu_{\sigma\circ u}.
		\eeq
		Note that since $\nu_R(D_R\setminus D'_R) =0$ by definition, we can replace the integrals on the left-hand side by integrals on $D_R$.
\end{enumerate}

Part 1 (the event in \eqref{eq:Main_distortion}) implies that $\dis \pi_{R_n}^{-1} < \e_n$. Part 4 (the event in \eqref{eq:Main_surjective}) implies that the $\e_n$-neighborhood of $\pi_{R_n}^{-1}(D'_{R_n})$ with respect to the Euclidean metric in $\R^d$ contains $D$, hence also with respect to $\boldd_{\eta\circ u}^D$. 
By applying $\pi_{R_n}$, it follows that the $\e_n$-neighborhood of $D'_{R_n}$ in $D_{R_n}= \pi_{R_n}(D)$ with respect to $d_{R_n}^D$ is indeed the whole $D_{R_n}$.
This shows that $\pi_R^{-1} : D'_{R_n}\to D$ is indeed an $\e_n$-approximation.

To prove \eqref{eq:measure_convergence_mu_R}, note that it suffices to prove it for Lipschitz functions on $D$.
Indeed, suppose we proved \eqref{eq:measure_convergence_mu_R} for Lipschitz functions, and let $f\in C(D)$. For every $\e>0$ there is a Lipschitz function $g$ on $D$ such that $\|f-g\|_\infty<\e$, hence 
\[
\begin{split}
& \limsup_n \left| \int_{D_{R_n}} f\circ \pi_{\revision{R_n}}^{-1} \,d\nu_\revision{R_n} - \int_D f \, d\mu_{\sigma\circ u} \right| \\
	&\quad \le \limsup_n \left| \int_{D_{R_n}} g\circ \pi_{\revision{R_n}}^{-1} \,d\nu_{\revision{R_n}} - \int_D g \, d\mu_{\sigma\circ u} \right| + \e\nu_R(D_{R_n}) + \e \mu_{\sigma\circ u}(D) \\
	&\quad \le 2\e \Leb_d(D).
\end{split}
\]
To prove \eqref{eq:measure_convergence_mu_R} for a Lipschitz function $f$, note that for $M>0$ large enough, $f/M\in W(D)$.
Therefore, when Part 2 (the event in \eqref{eq:Main_measure}) holds, 
\[
\left| \int_{D_{R_n}} f\circ \pi_{\revision{R_n}}^{-1} \,d\nu_{\revision{R_n}} - \int_D f \, d\mu_{\sigma\circ u} \right|
	\le M \e_n \to 0,
\]
which completes the proof. \hfill $\blacksquare$


\bibliographystyle{alpha}

\newcommand{\etalchar}[1]{$^{#1}$}
\def\cprime{$'$}

\end{document}